\documentclass[12pt]{amsart}

\usepackage{amsmath}
\usepackage{amsthm}
\usepackage{amssymb}
\usepackage{enumerate}
\usepackage{xcolor}
\usepackage{comment}

\usepackage{pst-node}
\usepackage{tikz-cd}

\usepackage[colorlinks=true]{hyperref}
\usepackage{hyperref}
\usepackage{cleveref}

\newtheorem{Theorem}{Theorem}[section]
\newtheorem{Lemma}[Theorem]{Lemma}
\newtheorem{Corollary}[Theorem]{Corollary}
\newtheorem{Proposition}[Theorem]{Proposition}
\newtheorem{Comments}[Theorem]{Comments}
\newtheorem{Remark}[Theorem]{Remark}

\newtheorem{Example}[Theorem]{Example}

\newtheorem{Definition}[Theorem]{Definition}
\newtheorem{Problem}[Theorem]{Problem}

\newtheorem{Question}[Theorem]{Question}

\newcommand{\trdeg}{\mbox{\rm trdeg}\,}

\newcommand{\depth}{\operatorname{depth}}

\newcommand{\spec}{\text{Spec}}

\newcommand{\sat}{\text{sat}}

\newcommand{\ass}{\text{Ass}}
\newcommand{\htt}{\text{ht}}

\renewcommand{\red}{\text{red}}

\newcommand{\gr}{\operatorname{gr}}
\newcommand{\grade}{\operatorname{grade}}

\title{Some Remarks about Saturation of Ideals}

\date{\today}

\begin{document}

\author[]{Souvik Dey}
\address[S. Dey]
{Department of Mathematical Sciences
850 West Dickson Street, University of Arkansas
Fayetteville, Arkansas 72701}
\email{souvikd@uark.edu}

\author[]{Stephen Landsittel}
\address[S. Landsittel]
{Institute of Mathematics, Hebrew University, Givat Ram, Jerusalem 91904, Israel}
\email{stephen.landsittel@mail.huji.ac.il}

\begin{abstract}
In this paper we observe when saturation of ideals in a local ring $R$ commutes with extension along ring maps and initial ideals. We give a characterization in terms of Cohen--Macaulayness for when this happens along the map $R\to R/\text{nil}(R)$. We give several examples and non examples where this happens, and we demonstrate an application of this condition to epsilon multiplicity. Additionally, we show that saturation commutes with extension along a flat injection $R\to S$ of local rings if and only if the closed fiber of the injection is Artinian.
\end{abstract}

\maketitle

\section{Introduction}


Given an algebraic variety $X$ defined by the vanishing of an ideal $I$ in a polynomial ring, the \emph{saturation} of $I$ is an enlarged version $I^{\sat}$ of $I$ which defines the same variety. From a contemporary standpoint of commutative algebra, the saturation of an ideal $I$ in a local ring $(R,\mathfrak{m})$ can be thought of as the asymptotic colon ideal
\begin{equation*}
    I^{\sat} := I:\mathfrak{m}^{\infty} = \cup_{i\geq 1}I:\mathfrak{m}^i.
\end{equation*}
If $I^{\sat} = I$, or equivalently, $H^0_\mathfrak{m}(R/I)=0$, then we say that $I$ is \emph{saturated}. Saturation has played a significant role in commutative algebra in the study of multiplicity theory, primary decomposition, and of other operations on ideals such as symbolic powers (see \cite{BenO1} and \cite[page 2]{CL}), and integral closure (see \cite{UV}) for instance. Recently saturation has been studied and computed in combinatorial contexts, see for instance \cite{SatEdge, SatMon}. Since saturation of an ideal is equivalent to removing the embedded component, see for instance Remark \ref{rmk-calc} (i), if $I$ is a graded ideal in a polynomial ring, $I^{\sat}$ is the largest ideal containing $I$ that defines the same projective scheme. Saturation is also discussed in a somewhat different, but also geometric light, in \cite{SatMon}.\\

Given local rings $R$ and $S$, we refer to an \emph{operation} $\mathcal{J}$ on ideals in $R$ as a mapping from ideals $I$ of $R$ to ideals $\mathcal{J}(I)$ in $S$. Often we are interested in cases where $I\subset \mathcal{J}(I)$ and $S = R$, e.g. integral closure, tight closure, radical, or saturation. In particular, saturation $I\mapsto I^{\sat}$ is a closure operation, see \cite{epstein} and Remark \ref{rmk-calc}. It would be convenient in studying multiplicity theory, primary decomposition, and in studying other operations on ideals and rings to understand how saturation behaves in relation to other operations on ideals. More specifically, we pose the following general problem which is of general interest across several areas of commutative algebra.

\begin{Problem}\label{problem-general_operations}
    Given a map $R\to S$ of local rings and an operation $\mathcal{J}$ from ideals $I$ in $R$ to ideals $\mathcal{J}(I)$ in $S$, find sufficient conditions (on $R$, $S$, and $\mathcal{J}$), or a characterization, for when saturation commutes with the operation, that is when $\mathcal{J}(I^{\sat}) = (\mathcal{J}(I))^{\sat}$ for all ideals $I$ in $R$.
\end{Problem}

Problem \ref{problem-general_operations} is quite broad and envelops a large sector of problems in commutative algebra. In the present work, we discuss Problem \ref{problem-general_operations} when $\mathcal{J}$ is extension along a wide variety of ring maps. More specifically, we investigate when saturation commutes with a given local ring map $R\to S$, or equivalently, when

\begin{equation}\tag{$\diamond$}\label{eq-main}
    I^{\sat}S = (IS)^{\sat}.
\end{equation}
The condition (\ref{eq-main}) always fails for $\mathfrak{m}$-primary ideals $I$ if $R\to S = R[[x_1,\ldots,x_n]]$ is the inclusion into a power series ring. However, we prove in three theorems (Theorems \ref{thmA}, \ref{thmB}, and \ref{thmC} below) that occurrence of the condition (\ref{eq-main}) is not uncommon, and it reveals a variety of properties of the map $R\to S$, of the rings $R$ and $S$, and properties of some of their numerical invariants. In other words, the occurrence of (\ref{eq-main}) (or lack thereof) carries a vast amount of information about the ring map $R\to S$. 

Theorem \ref{thmA} concerns the case where $R\to S$ is surjective, Theorem \ref{thmB} concerns the case where $R\to S$ is injective, and Theorem \ref{thmC} applies the condition (\ref{eq-main}) to a numerical invariant of Ulrich and Validashti called the epsilon multiplicity. If $R$ is a local domain, then for any local ring extension $R\subset S$, the condition (\ref{eq-main}) occurs at $I = (0)$ if and only if the dimension of $R$ is positive, see Corollary \ref{rmk-zerodim}. In Section \ref{sec-gr} we study the interaction between saturation and taking the initial form $I^*\subset \gr_{\mathfrak{m}}(R)$ of an ideal $I$ in a local domain $R$.\\

In Theorem \ref{thmA}, we show in part that if $R$ is Cohen--Macaulay, $S:= R_{\red}$, and
\begin{equation*}
    \begin{split}
        &\text{the condition (\ref{eq-main}) holds for all complete intersection ideals $I\subset R$}
    \end{split}
\end{equation*}
then $S$ is Cohen--Macaulay.

\begin{Theorem}\label{thmA} 
    Let $(R, \mathfrak m)$ be a Cohen--Macaulay local ring and let $J\subset R$ be an ideal. Let $S = R/J$ and let $t=\depth S$. Then the following conditions are equivalent.
    \begin{enumerate}
        \item[(I)] $S$ is Cohen--Macaulay and $\htt(J)=0$.
        \item[(II)]
        \ref{eq-main} holds for every ideal $I$ in $R$ such that $I$ is generated by $t$ elements which are regular on $R$ and $S$.
        \end{enumerate}
\end{Theorem}

We shall call a Noetherian scheme weakly reduced if all of the local rings of $X$ at closed points satisfy \ref{eq-main} at every ideal.
\begin{Corollary}
    Let $X$ be a Noetherian weakly reduced scheme. If $X$ is Cohen--Macaulay, then so is $X_{\red}$.
\end{Corollary}


The condition (II) happens plainly if $J = (0)$. We give examples of rings $R$ where the condition (II) does and does not happen when $J$ is the nilradical $N$ of $R$ as follows
\begin{enumerate}
    \item[(i)]In Remark 4.6 we show that (II) occurs when $\dim(R)=1$ and $R$ has a unique minimal prime.
    \item[(ii)]We give a class of $2$-dimensional rings which satisfy property (II) in Proposition \ref{prop-dim2}.
    \item[(iii)] In Proposition \ref{dim2thm2} we show that condition (II) is satisfied when $\dim(R)=2$ and the weaker statement $((g)+J)^{\sat}= (g)^{\sat}+J$ for every regular element $g$ on $R/J$ holds.
    \item[(iv)] In Example \ref{ex1} we provide an example of a complete intersection local ring for which property (II) fails
\end{enumerate}
We also show in Lemma \ref{lem2.0} that we can always complete $R$ in determining whether (II) holds. In Proposition \ref{proptest0} we show that we may enlarge the ideal $I:=(x_1,\ldots,x_t)$ of (II) to $I^{\sat}$ in computing whether $(I+J)^{\sat}=I^{\sat}+J$ (as in $I^{\sat}S = (IS)^{\sat}$) holds. We show in Example \ref{ex1} that there is a Cohen--Macaulay complete local ring $R$ such that $R/N$ is not Cohen--Macaulay and hence (II) can fail. Moreover, we describe a possible generalization of Theorem \ref{thmA} in Comments \ref{com1}, which we show fails in Example \ref{ex-depth2}.\\

When the map of local rings $R\to S$ we are studying is injective and flat, the equation (\ref{eq-main}) has additional characterizations. More specifically, in Theorem \ref{thmB} we show (in part) that if $R\to S$ is a flat extension of local rings, then (\ref{eq-main}) holds for all ideals $I\subset R$ if and only if the closed fiber of $R\to S$ has finite length. We say that an extension of local rings $(R,\mathfrak{m},k)\subset (S,\mathfrak{n},\ell)$ is unramified if $k\to \ell$ is separable and algebraic and $\mathfrak{m}S = \mathfrak{n}$. An extension of local rings is formally \'{e}tale if it is unramified and flat. Theorem \ref{thmB} shows that an extension $R\subset S$ being formally \'{e}tale is sufficient for equation (\ref{eq-main}) to hold at every ideal $I$ in $R$.

\begin{Theorem}\label{thmB}
    Let $(R,\mathfrak{m})\subset (S,\mathfrak{n})$ be an extension of local rings such that
    \begin{equation*}
        \text{(I:J)S = (IS):(JS) for all ideals $I,J\subset R$ (e.g. if $R\to S$ is flat).}
    \end{equation*}
    Recall the condition (\ref{eq-main}) on an ideal $I\subset R$, which states that $I^{\sat}S = (IS)^{\sat}$. The following statements are equivalent.
    \begin{enumerate}
        \item[(i)] (\ref{eq-main}) holds for all ideals $I$ in $R$.
        \item[(ii)](\ref{eq-main}) holds for some $\mathfrak{m}$-primary ideal $I$ in $R$. 
        \item[(iii)] $\sqrt{\mathfrak{m}S}=\mathfrak{n}$.
        \item[(iv)] For some $\mathfrak{m}$-primary ideal $I$ in $R$, the limit $\lim_{n\to\infty}\ell_S(((I^n)^{\sat}S/(I^nS)^{\sat})$ exists and equals zero. 
    \end{enumerate}
    In particular, (\ref{eq-main}) holds for all ideals $I$ in $R$ if the extension $R\subset S$ is \'{e}tale.
\end{Theorem}

Since a reduced Artinian local ring is a field, Theorem \ref{thmB} has the following corollary, which states a characterization of when an extension of local rings is unramified, in terms of saturation.

\begin{Corollary}
    Let $(R,\mathfrak{m},k)\subset (S,\mathfrak{n},l)$ be an extension of local rings such that the fiber $S/\mathfrak{m}S$ of $R\to S$ is reduced, $k\subset l$ is algebraic, and $k$ has characteristic zero. Then (\ref{eq-main}) holds for some $\mathfrak{m}$-primary ideal $I$ in $R$ if and only if the inclusion $R\subset S$ is unramified.
\end{Corollary}

Note that if $(R,\mathfrak{m})$ is a local ring and $\mathfrak{m}$ contains a nonzero divisor $y$ (for instance when $R$ is a domain of positive dimension), then for large $t$, $ (0)^{\sat}\subset (0):y^t = (0)$, as in $(0)^{\sat} = (0)$ so that (\ref{eq-main}) holds for $I = (0)$, where $S$ is any local ring extension of $R$. Thus the conclusions of Theorem \ref{thmB} are generally not equivalent to (\ref{eq-main}) holding for some ideal $I$.

In \cite{H1}, extensions of local rings, such that the generic fiber of the inclusion map is trivial, are studied. By  \cite[Proposition 2.2]{H1} and Theorem \ref{thmB} we obtain the following corollary concerning trivial generic fiber complete (TGF-complete) extensions, which gives sufficient conditions for the condition (\ref{eq-main}) to hold at every ideal.

\begin{Corollary}\label{corA}
    Let $R\subset S$ be a flat extension of complete local domains where $R$ is one-dimensional. If the contraction of every nonzero ideal of $S$ (to $R$) is nonzero (i.e. $R\to S$ is TGF-complete), then (\ref{eq-main}) holds for every ideal $I\subset R$.
\end{Corollary}

If $k\subset F$ are any fields, then  $k[[x]]\subset F[[x]]$ is TGF-complete (see  \cite[Remark 2.3]{H1}), so that this ring extension satisfies (\ref{eq-main}) at every ideal of $k[[x]]$. Some additional papers on TGF extensions are \cite{H2, TD}.

The condition (\ref{eq-main}) is also relevant to a certain numerical invariant. We prove in Theorem \ref{thmC} that the condition (\ref{eq-main}) also has application to a numerical invariant of ideals $I$ in a $d$-dimensional local ring $(R,\mathfrak{m})$, called the epsilon multiplicity. The epsilon multiplicity is denoted by $\varepsilon(I)$, and is defined as follows
\begin{equation*}
    \varepsilon(I):= \limsup_{n\to \infty}\frac{\ell(H^0_\mathfrak{m}(R/(I^n))}{n^d/d!}.
\end{equation*}
Epsilon multiplicity has been shown to describe integral dependence in \cite{UV}, and it was shown in \cite{CS} that the positivity of $\varepsilon(I)$ characterizes when $I$ has maximal analytic spread (i.e. when the analytic spread of $I$ equals the dimension of $R$). It was shown in \cite{Land1} that $\varepsilon(I)$ exists as a limit for any ideal $I$ as long as the nildradical $N$ of the $\mathfrak{m}$-adic completion of $R$ has nonmaximal dimension.\\

The following result shows that epsilon multiplicity can always be calculated in a reduced ring when it exists and equation (\ref{eq-main}) is satisfied.

\begin{Theorem}\label{thmC} (Theorem \ref{thm-spread2})
    Let $R$ be a Nagata local ring such that the nilradical of $R$ has nonmaximal dimension, and let $S = R_{\red}$. If (\ref{eq-main}) holds for all ideals in $R$, then
    \begin{equation*}
        \varepsilon(I) = \varepsilon(IR_{\red})
    \end{equation*}
    for any ideal $I\subset R$.
\end{Theorem}

\begin{Question}
    Let $R$ be a Nagata local ring such that the nilradical of $R$ has nonmaximal dimension, and let $S = R_{\red}$. Suppose that $R$ and $S$ are Cohen--Macaulay. By Theorem \ref{thmA}, (\ref{eq-main}) holds for $I$. Do we have that (\ref{eq-main}) holds for all powers of $I$? If so, then the proof of Theorem \ref{thmC} shows that $\varepsilon(I) = \varepsilon(IR_{\red})$. As in, we would have a version of Theorem \ref{thmC} where we do not need to assume that (\ref{eq-main}) holds for all ideals.
\end{Question}

Some additional papers on epsilon multiplicity are \cite{CHST,CL, SuD, DDRV,DRT}.\\

The structure of our paper is as follows. In Section \ref{sec1} we give some general remarks about saturation which will be useful in the proofs of the main theorems and are useful in general for calculating saturation, and we also prove Theorem \ref{thmB}. In Section \ref{sec2} we study the case of a ring map $R\to S$ when $S$ is $R$ modulo its nilradical, and we prove Theorem \ref{thmA}. In Section \ref{sec3} we recall some machinery regarding epsilon multiplicity from \cite{Land1} and we prove Theorem \ref{thmC}. In Section \ref{sec4} we discuss when the hypothesis of Theorem \ref{thmA} occur, and we give some conditions for when this happens. In Section \ref{sec5} we give some general conditions for when the hypothesis of Theorem \ref{thmA} occurs when the ambient ring $R$ has dimension two. In Section \ref{sec-seq}, we give an application of computing saturation to sequential Cohen-Macaulayness. Finally, in Section \ref{sec-gr} we study the behavior of saturation under the natural set map $R\to \gr_{\mathfrak{m}}(R)$.

\section{General properties of Saturation and Proof of Theorem \ref{thmB}}\label{sec1}

In this section we describe some properties of saturation which will be relevant to the proofs of the main theorems. In addition, we give some remarks which show that saturation detects positivity of the dimension of of a local ring. In particular, we show in Corollary \ref{rmk-zerodim} that a local domain has positive dimension if and only if the zero ideal is saturated. We also prove Theorem \ref{thmB} in this section.\\

Throughout this paper, local rings are assumed to be Noetherian. Whenever have a local ring $R$, $\mathfrak{m}_R$ will denote its maximal ideal unless stated otherwise. Extensions of local rings $R\to S$ are assumed to be local maps (i.e. $\mathfrak{m}_R\subset \mathfrak{m}_S$).\\

The following remark is useful for a variety of computations and proofs involving saturation.

\begin{Remark}\label{rmk-calc}
    Let $(R,\mathfrak{m})$ be a local ring and let $I\subset R$ be an ideal. Then we have the following.
    \begin{enumerate}
    \item[(i)] Let
    \begin{equation*}
        I = Q_1\cap \cdots \cap Q_r
    \end{equation*}be an irredundant primary decomposition of $I$. Then either $\mathfrak{m}\not\in \ass_R(I)$ and $I = I^{\sat}$, or
    \begin{equation*}
        I^{\sat} = Q_1\cap \cdots \cap \widehat{Q_i} \cap \cdots Q_r
    \end{equation*}
    where $Q_i$ is the $\mathfrak{m}$-primary component.
        \item[(ii)]
        $I^{\sat} = I:\mathfrak{m}^t$ for sufficiently large positive integers $t$.
        \item[(iii)] $I^{\sat}$ is not $\mathfrak{m}$-primary. Thus a saturated ideal in a local ring cannot be $\mathfrak{m}$-primary.
        \item[(iv)]The mapping $I\mapsto I^{\sat}$ defines a closure operation.
        \item[(v)]We have that $I^{\sat}=R$, if and only if, $I=R$ or $I$ is $\mathfrak{m}$-primary. 
    \end{enumerate}
\end{Remark}

Before proving of Remark \ref{rmk-calc}, we give some discussion on the utility of Remark \ref{rmk-calc} (ii). One might ask how high of a $t$ one needs to go up to in order to compute the saturation. The minimal $t$ for which $I:\mathfrak{m}^t = I:\mathfrak{m}^{t+1}$ is called the \emph{saturation number} of $I$. Saturation numbers are discussed and computed for instance in \cite{HGR, SatMon}.

\begin{proof}(i) and (ii) follow from  \cite[Lemma 1.2.2, 1.2.4]{Land2}. We will prove (iii), (iv), and (v). 

Now we prove (iii). Suppose that $I^{\sat}$ is $\mathfrak{m}$-primary. Then $I$ is not $\mathfrak{m}$-primary, since that would imply $I^{\sat} = R$. Thus $I$ is not saturated. So $I$ has an irredundant primary decomposition $I = Q_1\cap \cdots\cap Q_r$ where $\sqrt{Q_1} = m$, and consequently,
\begin{equation*}
    I^{\sat} = Q_2\cap \cdots \cap Q_r
\end{equation*}
which contradicts the fact that $I^{\sat}$ is $m$-primary (none of the $Q_i$ for $i>1$ are $\mathfrak{m}$-primary since the decomposition $I = Q_1\cap \cdots\cap Q_r$ is irredundant with $\sqrt{Q_1} = \mathfrak{m}$). This completes the proof of (iii).

Now we prove (iv). If $I\subset J$ are ideals in $R$, then the facts that $I^{\sat}\subset J^{\sat}$ and $I\subset I^{\sat}$ follow from the definition of saturation $I^{\sat}  = \cup_{j\geq 1}I:\mathfrak{m}^j$. The fact that $(I^{\sat})^{\sat} = I^{\sat}$ follows from (i).

Finally we prove (v). If $I=R$ or $I$ is $\mathfrak{m}$-primary, then $I^{\sat}=R$ by (ii). Conversely, if $I^{\sat} = R$, then (again by (ii)) let $t>0$ be a positive integer satisfying $R = I^{\sat} = I:\mathfrak{m}^t$. We have $1\in I:\mathfrak{m}^t$ so that $I$ contains a power of $m$. This completes the proof of the remark.
\end{proof}

When calculating multiplicities involving local cohomology, sometimes it is relevant whether $(0)^{\sat}$ is always the zero ideal. It holds when $R$ is a local ring and has a nonunit nonzero divisor. On the other hand, the following remark describes when it is possible that the zero ideal is not saturated. In Corollary \ref{rmk-zerodim} we characterize when $(0)^{\sat}=(0)$. Before proving Corollary \ref{rmk-zerodim}, we mention the following general lemma which can be useful for calculating saturation.

\begin{Lemma}\label{rmk-calc2}
    Let $(R,\mathfrak{m})$ be a local ring of positive dimension, then a radical ideal is saturated if and only if it is not $\mathfrak{m}$-primary.
\end{Lemma}
\begin{proof}
    Let $I\subset R$ be a radical ideal. Suppose that $I$ is saturated. If $I = R$ then we are done, so assume not. Then we have that $I$ is not $\mathfrak{m}$-primary by Remark \ref{rmk-calc} (v). Conversely, suppose that $I$ is not $m_R$-primary. If $I = R$ we are done, since $R^{\sat} = R$, so we can assume that $I\subset \mathfrak{m}$. So $I$ has a unique irredundant primary decomposition $I = P_1\cap\cdots\cap P_r$ where $P_1,\ldots,P_r$ are prime ideals. Since $I$ is not $\mathfrak{m}$-primary, at least one of the $P_i$ is not equal to $\mathfrak{m}$. Then by irredundancy of $I = P_1\cap\cdots\cap P_r$, none of $P_1,\ldots,P_r$ equal $\mathfrak{m}$. Consequently, $I$ is saturated by Remark \ref{rmk-calc} (i).
\end{proof}

We can find when the zero ideal is saturated in a general local ring via the following corollary.

\begin{Remark}\label{rmk-zerodim}Let $R$ be a local ring.
    \item[(1)] For any ideal $I\subset R$, $I^{\sat} = I$ if and only if $\depth (R/I)>0$. In particular, $\depth(R)>0$ if and only if $(0)^{\sat} = (0)$.
    \item[(2)] If $R$ has positive dimension, then the nilradical of $R$ is saturated. 
\end{Remark}
\begin{proof}

For (1), let $\mathfrak m$ be the maximal ideal of $R$. We notice that $H^0_{\mathfrak m}(R/I)=\dfrac{I^{sat}}{I}$. Since, depth$(R/I)>0$ if and only if $H^0_{\mathfrak m} (R/I)=0$, thus the claim follows.  

    Now we prove (2). Since $N = \sqrt{(0)}$, it is radical. Then since $\dim(R)>0$, $N$ is not $\mathfrak{m}_R$-primary. Now the result follows from Lemma \ref{rmk-calc2}.
\end{proof}

Thus if $R$ is a local domain of positive dimension with nilradical $N$, then we have that $(0)$ and $N$ are both saturated. Now we continue working towards proving Theorem \ref{thmB}. Next we will obtain the following two lemmas, the first of which implies that the limit in Theorem \ref{thmB} (iv) is well defined.

\begin{Lemma}
    Let $(R,\mathfrak{m})\mapsto (S,\mathfrak{n})$ be an extension of local rings and let $J$ be an ideal in $R$. If for sufficiently large integers $t$, $(JS):[(\mathfrak{m}S)^t]\subset (J:\mathfrak{m}^t)S$ (e.g. if $R\to S$ is flat), then $(JS)^{\sat}\subset J^{\sat}S$.
\end{Lemma}
\begin{proof}
    We can enlarge $t$ if necessary, so that $(JS)^{\sat} = (JS):(\mathfrak{n}^t)$, $J^{\sat} = \mathfrak{m}^t$, and $(JS)^{\sat}\subset J^{\sat}S$. Now we have
    \begin{equation*}
        (JS)^{\sat} = (JS):(\mathfrak{n}^t) \subset (JS):[(\mathfrak{m}S)^t]\subset (J:\mathfrak{m}^t)S = J^{\sat}S.
    \end{equation*}
\end{proof}

\begin{Lemma}\label{lem2.1}
Let $(R,\mathfrak{m})\subset (S,\mathfrak{n})$ be an extension of local rings such that $\sqrt{\mathfrak{m}S} = \mathfrak{n}$. Then for any ideal $I\subset S$ we have
\begin{equation*}
    I^{\sat}S\subset (IS)^{\sat}.
\end{equation*}
\end{Lemma}
\begin{proof}
    It is enough to show that we have the inclusion of sets $I^{\sat}\subset (IS)^{\sat}$. Let $a\in I^{\sat}$. Let $t>0$ be large enough that $I^{\sat} = I:\mathfrak{m}^r$ and $(IS)^{\sat} = (IS):\mathfrak{n}^r$ for $r\geq t$. Then we have
    \begin{equation*}
        a\mathfrak{m}^tS\subset IS.
    \end{equation*}
    On the other hand there is a positive integer $c$ such that $\mathfrak{n}^c\subset \mathfrak{m}S$. Thus
    \begin{equation*}
        a\mathfrak{n}^{tc}\subset a(\mathfrak{m}S)^t = a(\mathfrak{m}^t)S\subset IS
    \end{equation*}
    and we are done.
\end{proof}

Now we prove Theorem \ref{thmB}.

\begin{proof}
    Let $t>0$ be large enough that $I^{\sat} = I:\mathfrak{m}^r$ and $(IS)^{\sat} = (IS):\mathfrak{n}^r$ for $r\geq t$. If $(\mathfrak{m}S)^{\sat} = \mathfrak{m}^{\sat}S$, then
    \begin{equation*}
        S = RS = \mathfrak{m}^{\sat}S = (\mathfrak{m}S)^{\sat} = \mathfrak{m}S:\mathfrak{m}^t_S
    \end{equation*} so that $\mathfrak{m}^t_S\subset \mathfrak{m}S$, and thus $\sqrt{\mathfrak{m}S} = \mathfrak{n}$.

    Suppose on the other hand that $\sqrt{\mathfrak{m}S} = \mathfrak{n}$, take $c>0$ such that $\mathfrak{n}^c\subset \mathfrak{m}S$, and let $I\subset R$ be an ideal. Then
    \begin{equation*}
        \begin{split}
            I^{\sat}S = (I:\mathfrak{m}^t)S = (IS):(\mathfrak{m}^tS)
            \supset (IS):\mathfrak{n}^{tc} = (IS)^{\sat}
        \end{split}
    \end{equation*}
    so that (\ref{eq-main}) holds for $I$ by Lemma \ref{lem2.1}.

    So far we have proven that $(\mathfrak{m}S)^{\sat} = \mathfrak{m}^{\sat}S$ implies that $\sqrt{\mathfrak{m}S} = \mathfrak{n}$, and the latter implies that $(IS)^{\sat}= I^{\sat}S$ for all ideals in $R$, which certainly implies that $(\mathfrak{m}S)^{\sat} = \mathfrak{m}^{\sat}S$. Thus (i) and (iii) are equivalent. Now we will show that these two conditions are both equivalent to (ii). For that it is enough to fix an $\mathfrak{m}$-primary ideal $I\subset R$, assume that (\ref{eq-main}) holds for $I$, and show that (iii) holds. There exists $b>0$ such that $\mathfrak{m}^b\subset I$. Let $t>0$ be large enough that $t\geq b$, $I^{\sat} = I:\mathfrak{m}^t$, and $(IS)^{\sat} = IS:\mathfrak{n}^t$. Then we have
    \begin{equation*}
        I^{\sat} = I:\mathfrak{m}^t \supset \mathfrak{m}^b:\mathfrak{m}^t = R
    \end{equation*}so that
    \begin{equation*}
        S = I^{\sat}S = (IS)^{\sat} = IS:\mathfrak{n}^t.
    \end{equation*}Consequently, $\mathfrak{n}^t\subset IS\subset \mathfrak{m}S$, which implies (iii).

    Now we have shown that (i), (ii), and (iii) are equivalent. If (iv) holds, then for some $n>0$ and some $\mathfrak{m}$-primary ideal $I\subset R$ we have
    \begin{equation*}
        \ell_S((I^nS)^{\sat}/((I^n)^{\sat}S)=0.
    \end{equation*}
    As in, $(I^nS)^{\sat}=(I^n)^{\sat}S$, so that (ii) holds for the $\mathfrak{m}$-primary ideal $I^n$. On the other hand, if (i) holds, then certainly $\ell_S((\mathfrak{m}^nS)^{\sat}/((\mathfrak{m}^n)^{\sat}S)=0$ for all $n$, so that (iv) holds for the ideal $I:=\mathfrak{m}$. This completes the proof.
\end{proof}

\section{Proof of Theorem \ref{thmA}}\label{sec2}
In this section, we first prove some lemmas about saturation which will be useful in studying the condition of when $(I+N)^{\sat} = I^{\sat}+N$ occurs, where $I$ is an ideal in a local ring $R$ and $N$ is the nilradical of $R$. After establishing said lemmas, we prove Theorem \ref{thmA}.

If $R$ is any ring, then we will write $R_{\red}:= R/\sqrt{0}$. If $R$ is any local ring, then by a complete intersection ideal, we mean an ideal generated by a regular sequence $x_1,\ldots,x_i$, where $i\geq 1$. By an ideal generated by a length zero regular sequence, we simply mean the zero ideal.\\

We begin with some remarks which will be used to prove Theorem \ref{thmA}. We note how one of the inclusions in condition (II) of Theorem \ref{thmA} always holds in the following remark.

\begin{Remark}\label{rmk0}
    Let $R$ be a local ring, let $I,J\subset R$ be ideals, and let $S = R/J$. We have the following statements.
    \begin{enumerate}
        \item[(1)] $I^{\sat}+J\subset (I+J)^{\sat}$.
        \item[(2)] $I^{\sat}S=(IS)^{\sat}$ if and only if $(I+J)^{\sat}\subset I^{\sat}+J$.
        \item[(3)] $I^{\sat}S\subset (IS)^{\sat}$.
    \end{enumerate}    
\end{Remark}
\begin{proof}
    By Remark \ref{rmk-calc} we have $I^{\sat}\subset (I+J)^{\sat}$ and $J\subset J^{\sat}\subset (I+J)^{\sat}$, so that (1) holds.

    Now we prove (2), which, due to (1), is equivalent to the statement that: $I^{\sat}S=(IS)^{\sat}$ if and only if $(I+J)^{\sat} = I^{\sat}+J$. We have $I^{\sat}(R/J)=(I(R/J))^{\sat}$ if and only if
    \begin{equation*}
        \frac{I^{\sat}+J}{J} = \bigg(\frac{I+J}{J}\bigg)^{\sat} \bigg( = \frac{(I+J)^{\sat}}{J}\bigg)
    \end{equation*} (the second equality holds by \cite[Remark 2.1 (i)]{Land1}) if and only if $(I+J)^{\sat} = I^{\sat}+J$, which completes the proof.

    Now we prove (3). Again by \cite[Remark 2.1 (i)]{Land1}, we have $\frac{I^{\sat}+J}{J} = I^{\sat}S$ and $\frac{(I+J)^{\sat}}{J} = (IS)^{\sat}$. Consequently (3) holds by (1).
\end{proof}

\begin{Definition}
    Let $R$ be a local ring and let $I,J\subset \mathfrak{m}_R$ be ideals. For a positive integer $t$, we will say that $I$ is generated by a length $t$ regular sequence on $R/J$ if $I$ is generated by elements $x_1,\ldots,x_t\in R$ such that $x_1+J,\ldots,x_t+J$ is a regular sequence on $R/J$. By an ideal generated by a length zero regular sequence on an $R$-module $M$, we simply mean the zero ideal.
\end{Definition}

\begin{Remark}
    An ideal generated by a length $t$ regular sequence on $R/J$ is also generated by a length $t$ regular sequence on $R$ (by  \cite[Exercise 16.1]{Mat} and Krull's Height Theorem).
\end{Remark}

Now we prove three lemmas from which Theorem \ref{thmA} will follow. The first of which is stated below. We note that Ian Aberbach communicated a proof of Corollary \ref{prop1} to the authors, which we have modified to obtain the following lemma.

\begin{Lemma}\label{propCM}
    Let $R$ be a $d$-dimensional Cohen--Macaulay local ring. Let $J\subset R$ be an ideal and let $t = \depth (R/J)$. Suppose that
    \begin{equation*}
        \begin{split}
            &\textit{ (*):  for any complete intersection ideal $I\subset R$ of height exactly $t$ generated by $t$ elements }\\ & \textit{ which are regular on $R/J$, we have that }
            (I+J)^{\sat}\subset I^{\sat}+J.
        \end{split}
    \end{equation*}Then $R/J$ is Cohen--Macaulay.
\end{Lemma}

\begin{proof}
    Let $\mathfrak{m}$ be the maximal ideal of $R$. Suppose, if possible, that $R/J$ is not Cohen--Macaulay. Let $h = \htt (J)$.\\

    First we address the case when $t=0$, so assume this. We have by assumption that $J^{\sat} = (0)^{\sat}+J$. Since $R/J$ is not Cohen--Macaulay, we have that
    \begin{equation*}
        0=t<\dim R/J = d-\htt J\leq d
    \end{equation*}so that $d>0$. Then since $R$ is Cohen--Macaulay, we have that $\depth R/(0) = d>0$ so that $m\notin \ass_R R/(0)$, as in $(0)^{\sat}=(0)$. Consequently
    \begin{equation*}
        J^{\sat} = (0)^{\sat}+J=J.
    \end{equation*}
    But since $0=\depth R/J$, we have that $J^{\sat}\neq J$, a contradiction. Hence $R/J$ is Cohen--Macaulay.\\
    
    Now we may assume that $t>0$. By hypothesis we have
    \begin{equation*}
        t <\dim(R/J) = d-\htt(J)=d-h\leq d.
    \end{equation*}
    Now
    \begin{equation*}
        \depth_R(R\oplus (R/J)) = \min\{d,t\} = t.
    \end{equation*}Let $x_1,\ldots,x_t$ be a regular sequence on the $R$-module $R\oplus (R/J)$. Set $I = (x_1,\ldots,x_t)$, which is a complete intersection ideal of height $t<d-h\leq d$. Since $\depth(R/I)=d-t>0$, $m\notin \ass_R(R/I)$, so that $I^{\sat}=I$. Hence, $I^{\sat}+J=I+J$. However
    \begin{equation*}
        \depth(R/(I+J))
        = \depth ((R/J)/(I(R/J)))=0.
    \end{equation*}Thus $\mathfrak{m}\in \ass_R(R/(I+J))$. Hence
    \begin{equation*}
        (I+J)^{\sat}\neq I+J = I^{\sat}+J
    \end{equation*}which contradicts the hypothesis (*). Thus $R/J$ is Cohen--Macaulay.
\end{proof}

\begin{Corollary}\label{prop1}
    Let $R$ be a Cohen--Macaulay local ring and let $N$ be the nilradical of $R$. Suppose that
    \begin{equation*}
        \begin{split}
            &\textit{ (**): for any non-$\mathfrak{m}_R$-primary complete intersection ideal $I$, we have that}\\& \text{$(I+N)^{\sat} \subset I^{\sat}+N$.}
        \end{split}
    \end{equation*}Then $R_{\red}$ is Cohen--Macaulay.
\end{Corollary}
\begin{proof}
    Suppose that $R_{\red}$ is not Cohen--Macaulay and let $t = \depth R_{\red}$. Then $(I+N)^{\sat} \subset I^{\sat}+N$ holds for any complete intersection ideal of height $t$. Then $R_{\red}$ is Cohen--Macaulay by Lemma \ref{propCM}
\end{proof}

\begin{Lemma}\label{lem-posht}
    Let $R$ be a $d$-dimensional Cohen--Macaulay local ring and let $J\subset R$ be an ideal. Let $t=\depth R/J$. Suppose that $\htt(J)>0$ and let $M$ be the $R$-module $R\oplus (R/J)$. Then
        \begin{equation*}
            (I+J)^{\sat}\neq I^{\sat}+J
        \end{equation*}for any ideal $I\subset R$ generated by a length $t$ regular sequence on $M$.
\end{Lemma}
\begin{proof}
First we prove the result when $t = 0$. By definition, an ideal generated by a length zero regular sequence on $M$ is the zero ideal. So we must prove that
\begin{equation*}
    J^{\sat}\neq (0)^{\sat}+J.
\end{equation*}
We have that $0 = \depth R/J$ so that $J^{\sat}\neq J$. Note that $d>0$ since $\htt J>0$. Thus $0<d=\depth R/(0)$ so that $(0)^{\sat} = 0$. Hence $(0)^{\sat}+J= J\neq J^{\sat}$ and the lemma is proven in this case.

    Now we may assume that $t>0$. Always we have
    \begin{equation*}
        \depth(R\oplus (R/J))
        = \depth(R/J) = t.
    \end{equation*}
    Let $x_1,\ldots,x_t$ be a regular sequence on $M$ and let $I = (x_1,\ldots,x_t)$, which is complete intersection of height $t$ since $R$ is Cohen--Macaulay.\\

    Let $h = \htt J$. Note that if $d-t = 0$, we have $h=d-(d-h) = d-\dim(R/J)\leq d-\depth R/J = d-t = 0$, which is impossible as $h>0$. So $d-t>0$. As in, $\depth R/I>0$ so that $m$ is not associated to $R/I$. Thus, $I^{\sat} = I$.\\

    However, we always have
    \begin{equation*}
        \depth R/(I+J) = \depth ((R/J)/I(R/J)) = 0
    \end{equation*}so that
    \begin{equation*}
        (I+J)^{\sat} \neq I+J = I^{\sat}+J
    \end{equation*}which finishes the proof.
\end{proof}

\begin{Lemma}\label{lem-htzero}
    Let $R$ be a $d$-dimensional Cohen--Macaulay local ring and let $J\subset R$ be an ideal. Let $t=\depth R/J$. Suppose that $\htt(J)=0$ and that $R/J$ is Cohen--Macaulay. Then
    \begin{equation*}
            \begin{split}
            &\textit{ for every ideal $I$ generated by a length $t$ regular sequence on $R/J$ and $R$,} \\&(I+J)^{\sat}\subset I^{\sat}+J.
            \end{split}
        \end{equation*}
\end{Lemma}
\begin{proof}
    Let $\mathfrak{m}$ be the maximal ideal of $R$. Let $M$ be the $R$-module $R\oplus (R/J)$.\\

    First we address the case when $t=0$, so assume this. By definition the only ideal generated by a length $t$ regular sequence on $R/J$ is $(0)$, so we only need to prove that $J^{\sat}\subset (0)^{\sat}+J$. Since $R$ and $R/J$ are Cohen--Macaulay we have that
    \begin{equation*}
        0 = \depth R/J = \dim R/J = \dim R - \htt J = d.
    \end{equation*}Thus $\mathfrak{m} = \sqrt{0}$ so that $(0)$ is $\mathfrak{m}$-primary. Hence $(0)^{\sat} = R$ and we have that
    \begin{equation*}
        J^{\sat}\subset R =  (0)^{\sat}+J.
    \end{equation*}
    
    Now the lemma is proven when $t=0$. Suppose that $t>0$. Always $\depth_RM = \depth R/J = t$. Let $x_1,...,x_t$ be a regular sequence on $M$ and let $I = (x_1,\ldots,x_t)$ which is complete intersection of height $t$. Since $R$ and $R/J$ are Cohen--Macaulay,
    \begin{equation*}
        t=\depth R/J = \dim(R/J) = \dim(R)-\htt(J) = d-0=d
    \end{equation*}so that $I$ is generated by a system of parameters on $R$ and is hence $\mathfrak{m}$-primary. Thus,
    \begin{equation*}
        R = I^{\sat}\subset (I+J)^{\sat}
    \end{equation*}forcing $(I+J)^{\sat} = R$. Hence
    \begin{equation*}
        I^{\sat} + J = R=(I+J)^{\sat}
    \end{equation*}which finishes the proof.
\end{proof}

By combining Lemmas \ref{propCM}, \ref{lem-posht}, and \ref{lem-htzero}, the proof of Theorem \ref{thmA} is complete.



The difference between the depth and dimension in a local ring $R$ is called the Cohen--Macaulay defect, which was defined in \cite[D\'{e}finition 16.4.9]{EGA}. In Comments \ref{com1} and Example \ref{ex-depth2} we discuss this defect in relation to Theorem \ref{thmA}. Some papers on Cohen--Macaulay defect are \cite{CI, CI2, AFH}\\

\begin{Comments}\label{com1}
    A natural question would be whether we can extend Theorem \ref{thmA} to the case where $R$ is not Cohen--Macaulay by replacing condition (I) with the statement that $\depth R = \depth R/J$ and $\htt(J) = 0$. However, in Example \ref{ex-depth2} we show that there is a one-dimensional ring for which condition (II) in Theorem \ref{thmA} holds when $J$ is the nilradical $N$ of $R$, and $(\depth R/N = )\text{ } \depth R_{\red} >\depth R$.
\end{Comments}

We proceed with a definition motivated by Theorem \ref{thmA} and some comments regarding when the condition (II) of Theorem \ref{thmA} holds.

\begin{Definition}\label{df1}
    Let $R$ be a ring and let $N$ be the nilradical of $R$. We will say that $R$ is weakly reduced if (i) $R$ is local, and (ii) for every complete intersection ideal $I$ of $R$ we have that
    \begin{equation}\label{eq1.1}
        I^{\sat}+N= (I+N)^{\sat}.
    \end{equation}
    If instead of (ii), equation (\ref{eq1.1}) holds for one particular ideal $I$, then we will say that $R$ is weakly reduced at $I$, or that $I$ is a weakly reduced ideal.
\end{Definition}

We see that a local reduced ring is weakly reduced. Next we give an immediate application of the weakly reduced hypothesis.

\begin{Corollary}
    Let $R$ be a Cohen--Macaulay ring which is weakly reduced at every ideal. Then $R_{\red}$ is Cohen--Macaulay.
\end{Corollary}
\begin{proof}
    The result follows immediately from Corollary \ref{prop1}.
\end{proof}

By Remark 2.1 \cite{Land1}, equation (\ref{eq1.1}) is equivalent to saying that
\begin{equation*}
    (I(R/N))^{\sat} = I^{\sat}(R/N).
\end{equation*} By Remark \ref{rmk0} and Remark 2.1 \cite{Land1}, we always have that $I^{\sat}(R/N)\subset (I(R/N))^{\sat}$.\\

Using work of Hartshorne in \cite{RH} we obtain the following example, which shows that a complete intersection local ring need not be weakly reduced.
\begin{Example}\label{ex1}
    Let $F$ be an algebraic closure of $\mathbb{Z}/p\mathbb{Z}$ and let
    \begin{equation*}
        A = F[x,y,z,w]/(y^3+xz^2,x^8+w^2x^5y+w^4x^2y^2+w^6z^2).
    \end{equation*}Let $\mathfrak{M} = A_+$ be the graded maximal ideal and let $R = \widehat{A_{\mathfrak{M}}}$. Then $R$ is a complete intersection, while $R_{\red}$ is not Cohen--Macaulay. In particular, by Corollary \ref{prop1}, $R$ is a complete local ring which is not weakly reduced.
\end{Example}
The authors thanks Anurag Singh for bringing the ring $A$ in Example \ref{ex1} to their attention.\\

The preceding example shows that the $\depth(R_{\red})$ can be less than $\depth(R)$. A natural question would be to ask do we have for a local ring $R$ that $\depth R_{\red}\leq \depth R$. The next example show that this can fail even if $R$ is complete.

\begin{Example}\label{ex-depth}
    For all $n\geq 1$ there exists an $n$-dimensional depth zero complete local ring $(R,\mathfrak{m})$ such that $R_{\red}$ is Cohen-Macaulay, namely
    \begin{equation*}
        R:= k[[x_1,\ldots,x_n,x_{n+1}]]/(x_{n+1}^2,x_{n+1}x_1,\ldots,x_{n+1}x_n)
    \end{equation*}
    where $k$ is any field.
\end{Example}
\begin{proof}
    Fix a field $k$. First we show that $\depth R=0$. If $f\in \mathfrak{m}_R$, then $x_{n+1}f=0$, which shows that $\mathfrak{m}_R$ consists entirely of zero divisors.\\

    Next we show that $\depth R_{\red}=n$. We have
    \begin{equation*}
        \begin{split}
            R_{\red} &=
        (k[[x_1,\ldots,x_n,x_{n+1}]]/(x_{n+1}^2,x_{n+1}x_1,\ldots,x_{n+1}x_n))_{\red}
        \\&= \frac{k[[x_1,\ldots,x_n,x_{n+1}]]/(x_{n+1}^2,x_{n+1}x_1,\ldots,x_{n+1}x_n)}{(x_{n+1})/(x_{n+1}^2,x_{n+1}x_1,\ldots,x_{n+1}x_n)}
        =k[[x_1,\ldots,x_n]].
        \end{split}
    \end{equation*}This finishes the proof.
\end{proof}

\section{Proof of Theorem \ref{thmC}}\label{sec3}

Throughout the rest of this manuscript, we will denote the nilradical of a ring $A$ by $N(A)$. In this section we will prove Theorem \ref{thmB}, which we restate for convenience in Theorem \ref{thmC}.

\begin{Definition}\label{df-nagata} (\cite[pg. 264]{Mat})
    A domain $R$ with field of fractions $K$ is called Japanese if for any finite field extension $L$ of $K$, the integral closure of $R$ in $L$ is a finite $R$ module. We say that a Noetherian ring $R$ is Nagata if for every prime ideal $P$ of $R$ we have that $R/P$ is Japanese.
\end{Definition}

\begin{Example}
    Every a quasi-excellent ring is Nagata. In particular, complete local rings are Nagata.
\end{Example}

Recall the condition (\ref{eq-main}) from the introduction on an ideal $I$ in a local ring $R$ with nilradical $N$, in the case that $S:= R/N$, which states that
\begin{equation}\label{eq-main2}
    I^{\sat}(R/N) = (I(R/N))^{\sat}.
\end{equation}

Now we recall Theorem \ref{thmC}.

\begin{Theorem}\label{thm-spread2}
    Let $R$ be a Nagata local ring such that (\ref{eq-main2}) holds for every ideal $I$ in $R$. Let $d = \dim(R)$ and suppose that $\dim(N(R))<d$. Then $\varepsilon(I) = \varepsilon(IR_{\red})$ for any ideal $I$ of $R$.
\end{Theorem}

We will prove Theorem \ref{thm-spread2} by analyzing the limit of the following lemma.

\begin{Lemma}\label{rdlem}
(\cite[Proposition 2.6]{Land1})
    Let $R$ be a $d$-dimensional complete local ring and let $I\subset R$ be an ideal. Suppose that the nilradical $N$ of $R$ has dimension less than $d$. Then the limit
    \begin{equation*}
        \lim_{n\to\infty}\frac{\ell_R([(I^n+N)^{\sat}]/[(I^n)^{\sat}+N])}{n^d}
    \end{equation*}exists.
\end{Lemma}

We introduce a notation for a particular limit in the following definition, this definition will be convenient during the proof of Theorem \ref{thmC}.

\begin{Definition}\label{rddef}
Let $R$ be a $d$-dimensional local ring with nilradical $N$ and $I\subset R$ an ideal. Call
    \begin{equation*}
        DL(I):=\limsup_{n\to\infty}\frac{\ell_R([(I^n+N)^{\sat}]/[(I^n)^{\sat}+N])}{n^d}
    \end{equation*}the reduction defect limit of $I$ in $R$.
\end{Definition}

The motivation behind Definition \ref{rddef} comes from the following Remark. 

\begin{Remark}\label{rmk-dl}
    Let $R$ be a complete local ring such that the dimension of the $N(R)$ is less than $\dim(R)$. Then for any ideal $I$ in $R$
    \begin{equation*}
        DL(I) = \varepsilon(I R_{red})-\varepsilon(I).
    \end{equation*}
\end{Remark}
\begin{proof}
Let $N = N(R)$ and let $d=\dim(R)$. By \cite[page 5, proof of Lemma 2.4]{Land1}, for $n\geq 1$
\begin{equation*}
    \ell \left(\frac{(I^n)^{\sat}+N}{I^n+N}\right)
     = \ell\left((I^n)^{\sat})/I^n\right)
     -\ell\left(\frac{(I^n+N)\cap [(I^n)^{\sat}]}{I^n}\right)
\end{equation*} while  \cite[page 5, proof of Lemma 2.4]{Land1} also gives that for some $b>0$
\begin{equation*}
    \ell\left(\frac{(I^n+N)\cap [(I^n)^{\sat}]}{I^n}\right)/n^d
    \leq \ell(N/(N\mathfrak{m}^{nb})/n^d\to 0
\end{equation*}as $n\to \infty$ since $\dim(N)<d$. Then by equations (3) and (4) on \cite[page 5]{Land1} we have
    \begin{equation*}
        DL(I) = \varepsilon(I (R/N))- \varepsilon(I).
    \end{equation*}
\end{proof}

We will generalize Lemma $\ref{rdlem}$ to work in any Nagata ring (instead of complete). Before proving this, we need the following two lemmas.

\begin{Lemma}\label{lem-nagata}\emph{(}\cite[Tag 0331]{stacks-project}\emph{)}
    Let $R$ be a local Nagata domain. Then $R$ is analytically unramified.
\end{Lemma}

\begin{Lemma}\label{prop-nagata}
Let $R$ be a local ring.
\begin{enumerate}
    \item[(1)]$N(\widehat{R}) = N(R)\widehat{R}$ if and only if $R_{\red}$ is analytically unramified.
    \item[(2)] If $R$ is Nagata, then $R_{\red}$ is analytically\\ unramified.
\end{enumerate} 
\end{Lemma}
\begin{proof}
We first prove (1). The completion of $R_{\red}=R/N(R)$ is
    \begin{equation*}
        \widehat{R/N(R)} = \widehat{R}/N(R)\widehat{R}
    \end{equation*}and the nilradical of this ring is $N(\widehat{R})/N(R)\widehat{R}$, which is zero if and only if $N(\widehat{R}) = N\widehat{R}$. This completes the proof of (1).\\

    Now we prove (2). Let $S = R_{\red}$.
    By \cite[Proposition 9.1.3]{SH}, to show that $S$ is analytically unramified it is enough to fix a minimal prime $Q$ of $S$ and check that $S/Q$ is analytically unramified. But this follows from Definition \ref{df-nagata} and Lemma \ref{lem-nagata}. This completes the proof.
\end{proof}

We will now prove the following strengthening of \cite[Proposition 2.6]{Land1}.

\begin{Corollary}
    Let $R$ be a $d$-dimensional Nagata local ring with nilradical $N$ and let $I\subset R$ be an ideal. Suppose that the dimension of $N$ is less than $d$. Then the limit
    \begin{equation*}
        DL(I) = \lim_{n\to\infty}\frac{\ell_R([(I^n+N)^{\sat}]/[(I^n)^{\sat}+N])}{n^d}
    \end{equation*}exists and equals $DL(I\widehat{R})$.
\end{Corollary}
\begin{proof}
    Let $\mathfrak{m}$ be the maximal ideal of $R$. By Lemma \ref{prop-nagata}, $N\widehat{R} = N(\widehat{R})$. Fix $n\geq 1$ and let $J = I^n$. We can find $t$ large enough that $(J+N)^{\sat} = (J+N):\mathfrak{m}^t$ and $J^{\sat} = J:\mathfrak{m}^t$. Then we have $\mathfrak{m}^t(J+N)^{\sat}\subset J+N\subset J^{\sat}+N$ so that
\begin{equation*}
    \mathfrak{m}^t\frac{(J+N)^{\sat}}{J^{\sat}+N}=0.
\end{equation*}Now for $n\geq 1$, let $M_n = \frac{(I^n+N)^{\sat}}{(I^n)^{\sat}+N}$ and let $t_n\in\mathbb{Z}_{>0}$ such that $\mathfrak{m}^{t_n}M_n=0$. Since $M_n$ is an $R/\mathfrak{m}^{t_n}$-module, it is an Artin module. Then by  \cite[Lemma 2.11]{Land1}, $M_n$ is isomorphic as $R$-modules to the $R/\mathfrak{m}^{t_n}\cong \widehat{R/\mathfrak{m}^{t_n}}$-module $\widehat{M_n}$. Thus we have
\begin{equation*}
    \begin{split}
        \ell_{R}(M_n)
        = \ell_{R/\mathfrak{m}^{t_n}}(M_n)
        = \ell_{\widehat{R/\mathfrak{m}^{t_n}}}(M_n)
        = \ell_{\widehat{R}/\widehat{m}^{t_n}}(\widehat{M}_n)
        = \ell_{\widehat{R}}(\widehat{M}_n)
    \end{split}
\end{equation*}for all $n\geq 1$. Moreover, we have
\begin{equation*}
    \begin{split}
        \widehat{M_n}
        &=\frac{(I^n+N)^{\sat}\widehat{R}}{[(I^n)^{\sat}+N]\widehat{R}}
        =
        \frac{[(I^n+N)\widehat{R}]^{\sat}}{(I^n)^{\sat}\widehat{R}+N\widehat{R}}
        = \frac{[I^n\widehat{R}+N\widehat{R}]^{\sat}}{(I^n\widehat{R})^{\sat}+N\widehat{R}}
        \\&=
        \frac{[(I\widehat{R})^n+N(\widehat{R})]^{\sat}}{((I\widehat{R})^n)^{\sat}+N(\widehat{R})}.
    \end{split}
\end{equation*}Then by Lemma \ref{rdlem}, the limits
\begin{equation*}
    \begin{split}
        DL(I) = \lim_{n\to\infty}&\frac{\ell_R([(I^n+N)^{\sat}]/[(I^n)^{\sat}+N])}{n^d}
    \\&=\lim_{n\to\infty}\frac{\ell_{\widehat{R}}(([(I\widehat{R})^n+N(\widehat{R})]^{\sat})/[((I\widehat{R})^n)^{\sat}+N(\widehat{R})])}{n^d}
    = DL(I\widehat{R})
    \end{split}
\end{equation*}exist and are equal.
\end{proof}

\begin{Remark}\label{rmk1.1}
    Let $R$ be a $d$-dimensional complete local ring such that the dimension of the nilradical $N$ of $R$ is less than $d$. Then if $DL(I) = 0$ (e.g. if $R$ is weakly reduced at every power of $I$), we have $\varepsilon(I) = \varepsilon(I(R/N))$.
\end{Remark}
In particular, weak reducedness is enough that we can always compute epsilon multiplicities in a reduced setting.
\begin{proof}
    The result follows from Remark \ref{rmk-dl}.
\end{proof}

We observe a lemma which was shown in the proof of \cite[Corollary 5.6]{C2}.
\begin{Lemma}\label{lem-nilradical-nonmax}
    Suppose that $R$ is a local ring such that $N(\widehat{R}) = N(R)\widehat{R}$ (for instance, if $R$ is Nagata, by Lemma \ref{prop-nagata}). Then $\dim(N(\widehat{R})) = \dim(N(R))$.
\end{Lemma}
\begin{proof}
    Let $N = N(\widehat{R})$ and $M = N(R)$. We have that
    \begin{equation*}
        \dim_{\widehat{R}}(N) = \dim(\text{gr}_{\mathfrak{m}_R\widehat{R}}(N))
        =\dim(\text{gr}_{\mathfrak{m}_R\widehat{R}}(M\widehat{R}))
        = \dim(\text{gr}_{\mathfrak{m}_R}(M))
        =\dim_R(M)
    \end{equation*}which completes the proof.
\end{proof}

\begin{Proposition}\label{propeps}
    Let $R$ be a Nagata local ring such that $\dim(N(R))<\dim(R)$. Then
    \begin{equation*}
        \varepsilon(I) = \varepsilon(I(R/N(R)))
    \end{equation*}for any ideal $I\subset R$ such that the $DL(I\widehat{R})=0$.
\end{Proposition}

In Lemma \ref{lem2.0} we prove that the $DL(I\widehat{R})=0$ for any ideal $I$ if $R$ is weakly reduced. Now we prove Proposition \ref{propeps}.
\begin{proof}
    Let $N$ be the nilradical of $\widehat{R}$. Since $R$ is Nagata, we have that $N = N(R)\widehat{R}$ by Lemma \ref{prop-nagata}. Since $R$ is Nagata, we also have that $\dim(N(\widehat{R}))<\dim(R)$ by Lemma \ref{lem-nilradical-nonmax}. We now have that
    \begin{equation*}
        \varepsilon(I) = \varepsilon(I\widehat{R})
        = \varepsilon((I\widehat{R})(\widehat{R}/N))
    \end{equation*}
    by applying Remark \ref{rmk1.1} to the ring $\widehat{R}$ and the ideal $I\widehat{R}$. Let $A = \widehat{R}$. We see that
    \begin{equation*}
        \begin{split}
            (IA)(A/N)
            &= \frac{IA+N}{N}
            = \frac{IA+N(R)A}{N(R)A}
            =\{xy+N(R)A\mid x\in I,y\in A\}\\
            &= \{(x+N(R))(y+N(R)A)\mid x\in I,y\in A\}
            \\&=
            \frac{I+N(R)}{N(R)}\frac{A}{N}
            = (I(R/N(R)))(A/N)
            = (I(R/N(R)))(\widehat{R}/N(R)\widehat{R})\\
            &= (I(R/N(R)))\widehat{R/N(R)}.
        \end{split}
    \end{equation*}Thus,
    \begin{equation*}
        \varepsilon(I) = \varepsilon((I(R/N(R)))\widehat{R/N(R)})
        = \varepsilon(I(R/N(R)))
    \end{equation*}which completes the proof.
\end{proof}
Proposition \ref{propeps} allows us to reduce to a reduced setting in working with epsilon multiplicity in any Nagata ring whose nilradical has nonmaximal dimension.\\

In order to prove Theorem \ref{thm-spread2} we need one more lemma.

\begin{Lemma}\label{lem2.0}
    Let $R\to S$ be a flat extension of rings such that
    \begin{enumerate}
        \item[(i)] $N(S) = N(R)S$,
        \item[(ii)]  $R$ and $S$ are local, and
        \item[(iii)] $\mathfrak{m}_RS$ is $\mathfrak{n}$-primary.
    \end{enumerate}Then for any ideal $I$ of $R$, $I$ is weakly reduced if and only if $IS$ is a weakly reduced ideal of $S$.\\

    In particular, if $R$ is a Nagata local ring, then we can take $S = \widehat{R}$ by Lemma \ref{prop-nagata}.
\end{Lemma}
\begin{proof}
    Let $\mathfrak{m}$ be the maximal ideal of $R$. First we will prove that if $J\subset S$ is an ideal in $S$ then $J^{\sat}=J:(\mathfrak{m}S)^t$ for $t$ sufficiently large. Let $l$ be a positive integer such that $\mathfrak{n}^l\subset \mathfrak{m}S$ There is a positive integer $t$ such that $J^{\sat} = J:\mathfrak{n}^i$ for $i\geq t$. Now since $\mathfrak{n}^l\subset \mathfrak{m}S\subset \mathfrak{n}$ we have $\mathfrak{n}^{lt}\subset (\mathfrak{m}S)^t\subset \mathfrak{n}^t$ so that
    \begin{equation*}
        J^{\sat} = J:\mathfrak{n}^t
        \subset J:(\mathfrak{m}S)^t\subset J:\mathfrak{n}^{lt} = J^{\sat}
    \end{equation*}so that $J^{\sat} = J:(\mathfrak{m}S)^t$.\\

    Now we are ready to prove the conclusion of the lemma. Let $N = N(R)$. We have $N(S) = NS$ by assumption. We can now find a sufficiently large positive integer $t$ such that $I^{\sat} = I:\mathfrak{m}^t$, $(IS)^{\sat} = IS:(\mathfrak{m}S)^t$, $(I+N)^{\sat} = (I+N):\mathfrak{m}^t$, and $(IS+NS)^{\sat} = (IS+NS):(\mathfrak{m}S)^t$.\\

    Since $\mathfrak{m}S\subset \sqrt{\mathfrak{m}S} = \mathfrak{n}$, the map $R\to S$ is a flat local homomorphism. So $R\to S$ is faithfully flat. Hence we have
    \begin{equation}\label{eq-ff1}
        \begin{split}
            (I+N)^{\sat}S
            &= ((I+N):(\mathfrak{m}^t))S
            = (I+N)S : \mathfrak{m}^tS
            = (IS+NS):(\mathfrak{m}S)^t
            \\&= (IS+NS)^{\sat}
        \end{split}
    \end{equation}and
    \begin{equation}\label{eq-ff2}
        \begin{split}
            (I^{\sat}+N)S
            &=
            I^{\sat}S+NS
            = (I:\mathfrak{m}^t)S+NS
            =(IS:\mathfrak{m}^tS)+NS
            \\&=
            (IS)^{\sat}+NS.
        \end{split}
    \end{equation}If $R$ is weakly reduced at $I$, then extending the equation
    \begin{equation*}
        (I+N)^{\sat} = I^{\sat}+N
    \end{equation*}to $S$ implies that $S$ is weakly reduced at $IS$ by (\ref{eq-ff1}) and (\ref{eq-ff2}). If $S$ is weakly reduced at $IS$, then contracting the equation
    \begin{equation*}
        (IS+NS)^{\sat}=(IS)^{\sat}+NS
    \end{equation*}to $R$ yields that $(I+N)^{\sat} = I^{\sat}+N$ by (\ref{eq-ff1}), (\ref{eq-ff2}), and faithful flatness of $R\to S$.
\end{proof}

We are now ready to prove Theorem \ref{thm-spread2}.\\

\begin{proof}
    Let $R$ be a Nagata local ring which is weakly reduced at every ideal, and let $d = \dim(R)$. Let $I\subset R$ be an ideal. Now by Lemma \ref{lem2.0}, taking $S:= \widehat{R}$. We have that $IS = I\widehat{R}$ is weakly reduced for every ideal $I$ of $R$. Thus, $DL(I\widehat{R})$ exists as a limit and equals zero for every ideal $I$ of $R$. Theorem \ref{thm-spread2} now follows from Proposition \ref{propeps}.
\end{proof}

\section{Testing for the hypothesis of Theorem \ref{thmA}}\label{sec4}

In this section, we give some criteria and characterizations for when the conditions (I) (and equivalently, (II)), of Theorem \ref{thmA} hold. We begin with a remark and a short lemma.

Recall the condition (\ref{eq-main}) on a map $R\to S$ of local rings and an ideal $I$ in $R$ which states that
\begin{equation*}
    I^{\sat}S = (IS)^{\sat}.
\end{equation*}

\begin{Lemma}\label{lem5.1}
    Let $R$ be a local ring. If $I$ is any ideal in $R$, then
        \begin{equation*}
        (I^{\sat}+N(R))^{\sat} = (I+N(R))^{\sat}.
    \end{equation*}
\end{Lemma}
\begin{proof}
    Let $N$ be the nilradical of $R$. We prove (i). Since $I+N\subset I^{\sat}+N$, we have that\\ $(I+N)^{\sat}\subset (I^{\sat}+N)^{\sat}$. On the other hand, Remark \ref{rmk0} implies that
    \begin{equation*}
        \begin{split}
            (I^{\sat}+N)^{\sat}\subset [(I+N)^{\sat}]^{\sat}
            = (I+N)^{\sat}
        \end{split}
    \end{equation*}and thus the lemma holds.
\end{proof}

The following proposition shows that we can always assume that $I$ is saturated when testing for whether $I$ is weakly reduced.

\begin{Proposition}\label{proptest0}
    Let $R$ be a local ring. If $I$ and $J$ are two ideals of $R$ for which $I\subset J$ and $I^{\sat}=J^{\sat}$. Then $R$ is weakly reduced at $J$ if and only if $R$ is weakly reduced at $I$.
\end{Proposition}

\begin{proof}
    If $R$ is weakly reduced at $J$, then
    \begin{equation*}
        (I+N)^{\sat}\subset (J+N)^{\sat}\subset
        J^{\sat}+N = I^{\sat}+N
    \end{equation*}so that $R$ is weakly reduced at $I$. Suppose that $R$ is weakly reduced at $I$, then by Lemma \ref{lem5.1}
    \begin{equation*}
        \begin{split}
            (J+N)^{\sat}
            = (J^{\sat}+N)^{\sat}
            = (I^{\sat}+N)^{\sat}
            = (I+N)^{\sat}
            \subset I^{\sat}+N
            = J^{\sat}+N
        \end{split}
    \end{equation*}which finishes the proof of the proposition.
\end{proof}

\begin{Corollary}\label{test3}
    Let $R$ be a local ring then $R$ is weakly reduced if and only if $R$ is weakly reduced at every saturated complete intersection ideal which does not contain the nilradical.
\end{Corollary}

In the following comments we discuss some conditions for checking when the formula (\ref{eq-main}) holds and does not hold.

\begin{Comments}
    Let $R$ be a local ring. Let $N$ be the nilradical of $R$.
    \begin{enumerate}
        \item[(1)]One can ask whether $R$ is weakly reduced if and only if $R$ is weakly reduced at every saturated complete intersection radical ideal, however this is false in general. To see this, note that every radical ideal contains $\sqrt{0}=N$. And by definition $R$ is weakly reduced at every ideal $I$ containing $N$, since
    \begin{equation*}
        I^{\sat}+N = I^{\sat} = (I+N)^{\sat}.
    \end{equation*}
    \item[(2)] Furthermore, if $I\subset R$ is any ideal, then $(\sqrt{I})^{\sat} = \sqrt{I^{\sat}}$. In particular $(\sqrt{I})^{\sat}$ is radical, and is thus weakly reduced.
    \end{enumerate}
\end{Comments}
\begin{proof}
    We must prove (2). First we prove that for any ideals $A,B\subset R$ $\sqrt{A\cap B}=\sqrt{A}\cap \sqrt{B}$. We have that $\sqrt{A}\cap \sqrt{B}\supset \sqrt{A\cap B}$. In addition, if $f\in R$, $f^i\in A$ and $f^j\in B$, then $f^{\max\{i,j\}}\in A\cap B$, which proves our claim.\\
    
    Let $I\subset R$ be an ideal. To prove (2) it is enough to show that $\sqrt{I^{\sat}}= (\sqrt{I})^{\sat}$. If $I = R$ then both of these ideals clearly equal $R$. Let $\mathfrak{m}$ be the maximal ideal of $R$. If $\sqrt{I} = \mathfrak{m}$, then $I^{\sat} = R$ and $\sqrt{I} = \mathfrak{m}$ so that $\sqrt{I^{\sat}}= \sqrt{R} = R = \mathfrak{m}^{\sat}= (\sqrt{I})^{\sat}$. Hence we may assume that $I\subsetneq \mathfrak{m}$ is a not $\mathfrak{m}$-primary. Let $I = q_1\cap \ldots\cap q_s$ be an irredundant primary decomposition such that either $\mathfrak{m}$ is not an associated prime of $I$ and $I = I^{\sat}$ or $\sqrt{q_1} = \mathfrak{m}$. Let $p_i = \sqrt{q_i}$ for all $i$. If $\mathfrak{m}$ is associated to $I$ we have that
    \begin{equation*}
        \sqrt{I^{\sat}} = \sqrt{q_2\cap \ldots\cap q_s}
        =p_2\cap\ldots\cap p_s
        = (p_1\cap \ldots p_s)^{\sat}
        =(\sqrt{I})^{\sat}.
    \end{equation*}Now we may suppose that $\mathfrak{m}$ is not associated to $I$ and $I = I^{\sat}$. Since $\mathfrak{m}$ is not associated to $p_i$ for all $i$, and $p_1\cap \ldots \cap p_s$ is an irredundant primary decomposition of $\sqrt{I}$, we have
    \begin{equation*}
        \sqrt{I^{\sat}} = \sqrt{I}
        =p_1\cap \ldots \cap p_s
        = (p_1\cap \ldots \cap p_s)^{\sat}
        =(\sqrt{I})^{\sat}
    \end{equation*}which finishes the proof.
\end{proof}

In the following remark we give a sufficient condition for a one-dimensional local ring to be weakly reduced. Recall that a local ring $R$ with nilradical $N$ is called weakly reduced if the condition $I^{\sat}+N = (I+N)^{\sat}$ occurs at every ideal $I$ generated by a regular sequence in $R$. In the following remark we give sufficient conditions for a one dimensional ring to be weakly reduced.

\begin{Remark}\label{rmk-dim1}
    Let $R$ be a local ring of dimension one with a unique minimal prime ideal. Then $R$ is weakly reduced at every ideal, and hence $R$ is itself weakly reduced.
\end{Remark}
\begin{proof}
Let $\mathfrak{m}$ be the maximal ideal of $R$. Let $I\subset R$ be an ideal. we must show that $(I+N)^{\sat}\subset I^{\sat}+N$. If $\sqrt{I}\supset \mathfrak{m}$, then $I^{\sat} = R$, and we are done by taking $J = I^{\sat}$ in Proposition \ref{proptest0} since $(R+N)^{\sat} = R = R^{\sat}+N$. So we may assume that $\sqrt{I}\subsetneq \mathfrak{m}$, and by Proposition \ref{proptest0} we may assume that $I$ is saturated. $I$ has an irredundant primary decomposition
    \begin{equation*}
        I = q_1\cap \ldots\cap q_s.
    \end{equation*}Since $I$ is saturated and not $\mathfrak{m}$-primary, we must have that $\sqrt{q_i}\neq \mathfrak{m}$ for all $i$. Since $\dim(R)=1$ and $R$ has a unique minimal prime, this minimal prime must be $N$ and we have $\spec(R) = \{\mathfrak{m},N\}$. Then by irredundancy of $I = q_1\cap \ldots\cap q_s$, we have $s=1$ and $\sqrt{q_1} = N$. Thus, $I = q_1\subset N$ so that
    \begin{equation*}
        (I+N)^{\sat} = N^{\sat} = N = I+N = I^{\sat}+N
    \end{equation*}which finishes the proof.
\end{proof}

Now we show that there are one dimensional local rings (with even just two minimal primes) which are not weakly reduced.

\begin{Example}\label{ex-2min}
    Let $R = \mathbb{C}[[x,y]]/(x^2y)$. Then $R$ is not weakly reduced. In fact $R$ is not weakly reduced at the complete intersection ideal $I:= (\overline{x}^2)$ (where $\overline{f}$ is the image of a polynomial $f\in \mathbb{C}[[x,y]]$ modulo $(x^2y)$).
\end{Example}
\begin{proof}
    The ideal $I$ is saturated since it is $(\overline{x})$-primary. Also the nilradical $N$ of $R$ is $N = (\overline{xy})$. But then
    \begin{equation*}
        I^{\sat}+N = I+N = (\overline{x}^2,\overline{xy})
        =(\overline{x})\cap (\overline{x}^2,\overline{y})
    \end{equation*}is an irredundant primary decomposition of $I+N$ so that
    \begin{equation*}
        (I+N)^{\sat} = ((\overline{x})\cap (\overline{x}^2,\overline{y}))^{\sat}
        = (\overline{x}) \neq
        (\overline{x}^2,\overline{xy})
        = I^{\sat}+N
    \end{equation*}so that $I$ is not weakly reduced. Consequently $R$ is not weakly reduced, and this finishes the proof.
\end{proof}

In Theorem \ref{thmC} we showed that the condition $I^{\sat}R_{\red} = (IR_{\red})^{\sat}$ (on ideals $I$ in $R$) controls the epsilon multiplicity, and in Theorem \ref{thmA} we showed that when $R$ is Cohen--Macaulay, this condition forces $R_{\red}$ to be Cohen--Macaulay. In the following example we show that this condition need not control the depth of $R_{\red}$ when $R$ is not Cohen--Macaulay.

\begin{Example}\label{ex-depth2}
    There exists a one-dimensional weakly reduced depth zero complete local ring $R$ such that $R_{\red}$ is Cohen-Macaulay, namely $R:= k[[x,y]]/(x^2,xy)$ where $k$ is any field.
\end{Example}
\begin{proof}
    By Example \ref{ex-depth} it remains to show that $R$ is weakly reduced. But $R$ is one dimensional and since $\sqrt{(x^2,xy)R} = xR$, $R$ has a unique minimal prime. Now we are done by Remark \ref{rmk-dim1}.
\end{proof}

\section{Dimension 2}\label{sec5}

In this section we give further criteria and characterizations for when the hypothesis of Theorem \ref{thmA} hold (in particular that $R_{\red}$ is Cohen--Macaulay), specifically when the ring we are studying has dimension two. We give such criteria in Proposition \ref{dim2thm2} and Proposition \ref{prop-dim2}. In Proposition \ref{dim2thm2} we show that the requirement for saturation of ideals generated by a regular sequence to commute with extension to $R_{\red}$ can be weakened to just asking for the condition on ideals generated by a single regular element. In Proposition \ref{dim2thm2} we state an explicit class of rings which satisfy the hypothesis of Theorem \ref{thmA} (as in that saturation commutes with extension to $R_{\red}$ for all ideals generated by a regular sequence).

\begin{Remark}
    Let $R$ be a local ring with nilradical $N$. Note that if $I\subset R$ is an $\mathfrak{m}_R$-primary ideal, say $\mathfrak{m}_R^l\subset I$. Then $\mathfrak{m}_R^l\subset I+N(R)$ and thus
    \begin{equation*}
        (I+N)^{\sat} = R = R+N = I^{\sat}+N
    \end{equation*}so that $I$ is weakly reduced. Thus if $I\subset R$ is an ideal which is not weakly reduced, we must have $N\not\subset I$ and $\sqrt{I}\neq \mathfrak{m}_R$, so that $\sqrt{I}\subsetneq \mathfrak{m}_R$.
\end{Remark}

\begin{Proposition}\label{dim2thm2}
    Let $R$ be a 2-dimensional local Cohen--Macaulay ring and let $N$ be the nilradical of $R$. If
    \begin{equation*}
        (***): ((g)+N)^{\sat} = (g)+N \text{ for every $R$-regular element $g\in R$}
    \end{equation*}then $R_{\red}$ is Cohen--Macaulay.
\end{Proposition}
\begin{proof}
    Let $\mathfrak{m}$ be the maximal ideal of $R$. We have that
    \begin{equation*}
        \ass_R(R/N) = \text{Min}(R/N)
        = \text{Min}(R)
    \end{equation*}does not contain $\mathfrak{m}$, since $\dim(R)>0$. Thus, $\depth (R/N) >0$.\\

    Suppose, if possible, that $R/N$ is not Cohen--Macaulay. Then $t:= \depth (R/N) = 1$. Let $g\in R$ be a (necessarily maximal) regular sequence on $R\oplus R/N$ (as $\depth_R R\oplus (R/N) = \min\{\depth R, \depth R/N\}=1$). Set $I = gR$. We have that $\depth R/I = 1$ (as $\depth R = 2$), so that $\mathfrak{m}\notin \ass (R/I)$. Thus $I = I^{\sat}$, while
    \begin{equation*}
        \depth(R/(I+N))
        =\depth ((R/N)/(I(R/N))) = 0
    \end{equation*} and hence
    \begin{equation*}
        (I+N)^{\sat}\neq I+N = I^{\sat}+N
    \end{equation*}which contradicts our assumption (***). Hence, $R/N$ is Cohen--Macaulay.
\end{proof}

\begin{Remark}\label{reg-rmk}
    Let $(A,\mathfrak{m})$ be a Cohen--Macaulay local ring and let $p\in \{1,\ldots,d\}$ and let $n_1,\ldots,n_p$ be positive integers. If $\mu_1,\ldots,\mu_p\in \mathfrak{m}$. Then $\mu_1,\ldots,\mu_p$ is a regular sequence if and only if $\mu_1^{n_1},\ldots,\mu_p^{n_p}$ is.
\end{Remark}
\begin{proof}
    The ideals $(\mu_1^{n_1},\ldots,\mu_p^{n_p})$ and $(\mu_1,\ldots,\mu_p)$ have the same height. Let $h$ be this height. Then since $A$ is Cohen--Macaulay, $\mu_1^{n_1},\ldots,\mu_p^{n_p}$ is a regular sequence if and only if $h = p$ if and only if $\mu_1^{n_1},\ldots,\mu_p^{n_p}$ is a regular sequence.
\end{proof}

\begin{Lemma}\label{satlemma}
    Let $(A,\mathfrak{m})$ be a local ring, $J\subset A$ an ideal, and $Q\subset A$ an $\mathfrak{m}$-primary ideal. If $Q = (f_1,\ldots,f_s)$, then we have
    \begin{equation*}
        J^{\sat} = \bigcap_{i=1}^s(J:(f_i^t))
    \end{equation*}for sufficiently large positive integers $t$.
\end{Lemma}
\begin{proof}
    The families $\{\mathfrak{m}^i\mid i\geq 1\}$ and $\{Q^i\mid i\geq 1\}$ are cofinal, and hence the families $\{\mathfrak{m}^i\mid i\geq 1\}$ and $\{(f_1^i,\ldots,f_s^i)\mid i\geq 1\}$ are cofinal. Thus for sufficiently large $t$,
    \begin{equation*}
        J^{\sat}
        = J:(f_1^t,\ldots,f_s^t)
        =\bigcap_{i=1}^s(J:(f_i^t)).
    \end{equation*}
\end{proof}

\begin{Proposition}\label{prop-dim2}
    Let $A$ be a Cohen--Macaulay local ring of dimension $d\geq 2$ let $\mu_1,\ldots,\mu_{d-2}$ be an $A$-regular sequence such that $(\mu_1,\ldots,\mu_{d-2})$ is a radical ideal, and let $n_1,\ldots,n_{d-2}\geq 2$ be positive integers. Let $R = A/(\mu_1^{n_1},\ldots,\mu_{d-2}^{n_{d-2}})$. Then $R$ is weakly reduced.
\end{Proposition}
\begin{proof}
    Firstly, $\mu_1^{n_1},\ldots,\mu_{d-2}^{n_{d-2}}$ is an $A$-regular sequence, and hence $R$ is a 2-dimensional Cohen--Macaulay ring. For $f\in A$, let $\overline{f}$ be the image of $f$ in $R$.\\

    Let $I\subset R$ be a complete intersection ideal. We may assume that $I$ has one generator (otherwise $I$ would be $\mathfrak{m}_R$-primary and hence weakly reduced). Write $I = (\overline{x})$ where $x\in A$ and $\overline{x}$ is $R$-regular. We have that $\depth R/(\overline{x}) = 1$ so that $\mathfrak{m}\notin \ass (R/(\overline{x}))$. Thus $(\overline{x})^{\sat} = (\overline{x})$. Furthermore, the nilradical $N$ of $R = A/(\mu_1^{n_1},\ldots,\mu_{d-2}^{n_{d-2}})$ is $(\mu_1,\ldots,\mu_{d-2})/(\mu_1^{n_1},\ldots,\mu_{d-2}^{n_{d-2}}) = (\overline{\mu_1},\ldots,\overline{\mu_{d-2}})$, since $(\mu_1,\ldots,\mu_{d-2})$ is a radical ideal of $A$.\\

    Since $I^{\sat}=I = (\overline{x})$ and $N = (\overline{\mu_1},\ldots,\overline{\mu_{d-2}})$, it remains to show that
    \begin{equation*}
        (\overline{x},\overline{\mu_1},\ldots,\overline{\mu_{d-2}})^{\sat}
        = (\overline{x},\overline{\mu_1},\ldots,\overline{\mu_{d-2}}).
    \end{equation*}

    By construction, $\mu_1^{n_1},\ldots,\mu_{d-2}^{n_{d-2}},x$ is an $A$-sequence. Then $\mu_1,\ldots,\mu_{d-2},x$ is also a regular sequence in $A$ by Remark \ref{reg-rmk}. Then since $A$ is Cohen--Macaulay of dimension $d$, there exists $\alpha\in A$ such that $\mu_1,\ldots,\mu_{d-2},x,\alpha$ is a system of parameters in $A$. The ideal $Q := (\mu_1,\ldots,\mu_{d-2},x,\alpha)\subset A$ is $m_A$-primary. Now for any ideal $J\subset A$, we have that
    \begin{equation*}
        J^{\sat}
        =(J:\mu_1^t)\cap (J:\mu_2^t)\cap \ldots\cap 
        (J:\mu_{d-2}^t)\cap (J:x^t)
    \end{equation*}for sufficiently large positive integers $t$. Let $\mu_{d-1}=x$ and $\mu_d = \alpha$. Enlarge $t$ if necessary so that $t\geq n_i$ for $1\leq i\leq d-2$. Then by \cite[Remark 2.1]{Land1} and Lemma \ref{satlemma}
    \begin{equation*}
        \begin{split}
            \bigg(\frac{J+ (\mu_1^{n_1},\ldots,\mu_{d-2}^{n_{d-2}})}{(\mu_1^{n_1},\ldots,\mu_{d-2}^{n_{d-2}})}&\bigg)^{\sat}
            =
            \frac{(J+ (\mu_1^{n_1},\ldots,\mu_{d-2}^{n_{d-2}}))^{\sat}}{(\mu_1^{n_1},\ldots,\mu_{d-2}^{n_{d-2}})}
            \\&=
            \frac{\cap_{i=1}^d[(J+ (\mu_1^{n_1},\ldots,\mu_{d-2}^{n_{d-2}})):\mu_i^t]}{(\mu_1^{n_1},\ldots,\mu_{d-2}^{n_{d-2}})}
            \\&=
            \frac{\cap_{i=d-1}^d[(J+ (\mu_1^{n_1},\ldots,\mu_{d-2}^{n_{d-2}})):\mu_i^t]}{(\mu_1^{n_1},\ldots,\mu_{d-2}^{n_{d-2}})}
            \\&=
            \frac{[(J+ (\mu_1^{n_1},\ldots,\mu_{d-2}^{n_{d-2}})):x^t]\cap [(J+ (\mu_1^{n_1},\ldots,\mu_{d-2}^{n_{d-2}})):\alpha^t]}{(\mu_1^{n_1},\ldots,\mu_{d-2}^{n_{d-2}})}
        \end{split}
    \end{equation*}Now letting $J = (x,\mu_1,\ldots,\mu_{d-1})\subset A$ we see that
    \begin{equation*}
        \bigg(\frac{(x,\mu_1,\ldots,\mu_{d-2})}{(\mu_1^{n_1},\ldots,\mu_{d-2}^{n_{d-2}})}\bigg)^{\sat}
        =\frac{(x,\mu_1,\ldots,\mu_{d-2}):\alpha^t}{(\mu_1^{n_1},\ldots,\mu_{d-2}^{n_{d-2}})}
    \end{equation*}
    as $(x,\mu_1,\ldots,\mu_{d-2}):x^t = R$. Since $x,\mu_1,\ldots,\mu_{d-1},\alpha$ is an $A$-regular sequence, so is $x,\mu_1,\ldots,\mu_{d-1},\alpha^t$. Thus, $(x,\mu_1,\ldots,\mu_{d-2}):\alpha^t = (x,\mu_1,\ldots,\mu_{d-2})$. Therefore
    \begin{equation*}
\begin{split}
    (I+N)^{\sat}
    &=(\overline{x},\overline{\mu_1},\ldots,\overline{\mu_{d-2}})^{\sat}
    =\bigg(\frac{(x,\mu_1,\ldots,\mu_{d-2})}{(\mu_1^{n_1},\ldots,\mu_{d-2}^{n_{d-2}})}\bigg)^{\sat}
    \\&=
    \frac{(x,\mu_1,\ldots,\mu_{d-2})}{(\mu_1^{n_1},\ldots,\mu_{d-2}^{n_{d-2}})}
    = (\overline{x},\overline{\mu_1},\ldots,\overline{\mu_{d-2}})
    = I+N
    = I^{\sat}+N
\end{split}
    \end{equation*}so that $I$ is weakly reduced. This completes the proof.
\end{proof}

\section{Saturation and Sequentially Cohen--Macaulay Rings}\label{sec-seq}

In this section we discuss how computing saturation informs sequential Cohen-Macaulayness of a new class of rings. We begin by recalling some background about sequentially Cohen--Macaulay modules from \cite{CSSS}.

\begin{Definition}\label{def-seq} (\cite[Definition 1.1]{CSSS})
    Let $R$ be either a local ring or a standard graded algebra over a field. A finitely generated $R$-module $M$ is called sequentially Cohen-Macaulay if it admits a filtration of submodules $0=M_0\subsetneq M_1,\ldots,\subsetneq M_r=M$ such that each quotient
of the filtration $M_i/M_{i-1}$ is Cohen-Macaulay and $\dim(M_i/M_{i-1}) < \dim(M_{i+1}/M_i)$ for all $i$. Such a filtration is called an \textit{sCM filtration} for M.
\end{Definition}

\begin{Definition}(\cite[Definition 1.7]{CSSS})
    Let $k$ be a field, let $n\geq 1$, and let $S = k[x_1,\ldots,x_n]$ be the polynomial ring. A monomial ideal $I\subset S$ is said to be weakly stable if for all monomials
$u\in I$ and all integers $i,j$ with $1\leq j<i\leq n$, there exists $t\in \mathbb{N}$ such that $x_j^tu/x_i^l\in I$, where $l$ is the largest integer such that $x_i\mid u$.
\end{Definition}

\begin{Remark}(\cite[Proposition 1.8]{CSSS})
    Let $k$ be a field, let $n\geq 1$, and let $S = k[x_1,\ldots,x_n]$ be the polynomial ring. A monomial ideal $I\subset S$ is weakly stable if and only if all its associated primes are of the form $(x_1,\ldots,x_i)$ for some $i$.
\end{Remark}

\begin{Proposition}(\cite[Proposition 1.8]{CSSS})
    Let $k$ be a field, let $n\geq 1$, and let $S = k[x_1,\ldots,x_n]$ be the polynomial ring. and let $I\subset S$ be a weakly stable ideal. Then $S/I$ is sequentially Cohen--Macaulay
\end{Proposition}

In \cite[Proposition 1.8]{CSSS} it is shown that a polynomial ring modulo a weakly stable ideal is sequentially Cohen--Macaulay. Now we present a class of examples of (graded) sequentially Cohen--Macaulay rings which are not covered by \cite[Proposition 1.8]{CSSS}.

\begin{Proposition}\label{prop-seq}
    Let $k$ be a field, let $n\geq 2$, and let $S = k[x_1,\ldots,x_n]$ be the polynomial ring. Let $\mathfrak{M}$ be the graded maximal ideal of $S$. Let $f\in S$ be a nonconstant monomial which is not divisible by all of $x_1,\ldots,x_n$. Let $G\subset S$ be a nonzero proper monomial ideal, and let $I = fG\subset S$. Take $R = S/I$. Then
    \begin{enumerate}
        \item[(i)]If $\sqrt{G} = \mathfrak{M}$, then $R$ is sequentially Cohen--Macaulay, in fact
    \begin{equation*}
        0\to H^0_\mathfrak{M}(R)\to R\to 0
    \end{equation*} is a sCM filtration for $R$.
    \item[(ii)] $I$ is not weakly stable.
    \end{enumerate}
\end{Proposition}
\begin{proof}
    By assumption, we may assume without loss of generality that $x_1\nmid f$. We first prove (i). Let $t\in\mathbb{N}$ such that $I^{\sat}=I:\mathfrak{M}^t$. We begin by showing that $I^{\sat} = fS$. First we prove the case when $G = \mathfrak{M}^i$ for some $i\geq 1$. After enlarging $t$ if necessary, we may assume that $t\geq i$. Then we have $f\in f\mathfrak{M}^i:\mathfrak{M}^t = I^{\sat}$, so that $fS\subset I^{\sat}$. Let $H\in I^{\sat} = I:\mathfrak{M}^t$. On the other hand, $x_1^tH\in \mathfrak{M}^t(I:\mathfrak{M}^t)\subset I = f\mathfrak{M}^i$, so that $f\mid H$. Hence $fS = I^{\sat}$. Now we prove the identity $I^{\sat}=I:\mathfrak{M}^t$ in the general case. Since $\sqrt{G}=\mathfrak{M}$ we have that $\mathfrak{M}^j\subset G$ for some $j\geq 1$. So we have $fS = (f\mathfrak{M}^j)^{\sat}\subset (fG)^{\sat} = I^{\sat}$. On the other hand, $I^{\sat} = (fG)^{\sat}\subset (f\mathfrak{M})^{\sat} = fS$ since $G\subset \mathfrak{M}$. Now we have proven that $I^{\sat}=I:\mathfrak{M}^t$ and we are ready to finish the proof of (i). We calculate $H^0_\mathfrak{M}(R)$ as follows
\begin{equation*}
    \begin{split}
        H^0_\mathfrak{M}(R)&=H^0_\mathfrak{M}(S/(x_1f,\ldots,x_nf))
        = (x_1f,\ldots,x_nf)^{\sat}/(x_1f,\ldots,x_nf)\\
        &= fS/(x_1f,\ldots,x_nf).
    \end{split}
\end{equation*}
Thus
\begin{equation*}
    R/H^0_\mathfrak{M}(R)
    = \frac{S/(x_1f,\ldots,x_nf)}{fS/(x_1f,\ldots,x_nf)}
    \cong S/fS
\end{equation*}is a hypersurface, and is thus Cohen--Macaulay. This finishes the proof of (i).

Now we prove (ii). First observe that since $f$ is a monomial and $G$ is a monomial ideal, $fG$ is a monomial ideal, so that
\begin{equation*}
    \begin{split}
        (\dagger): &\text{ a monomial $F\in S$ lies in $fG$ if and only if $F$ is divisible by }
    \\&\text{one of the monomial generators of $fG$.}
    \end{split}
\end{equation*}

Since $f$ is a nonconstant monomial, we can assume without loss of generality that $x_n\mid f$. Fix $s\in\mathbb{N}$ and let $H\in G$ be a nonzero monomial. Write $H= x_1^{b}H_0x_n^c$ where $b,c\in\mathbb{N}$. To prove that $I$ is not weakly stable, it suffices to show that
\begin{equation*}
    M:= x_1^s(Hf)/x_n^{\text{ord}_{x_n}(Hf)}\notin I
\end{equation*}where $\text{ord}_{x_n}(Hf)$ is the highest power of $x_n$ that divides $Hf$, and so $\text{ord}_{x_n}(Hf) = c + \text{ord}_{x_n}(f)$. We have that $x_n\nmid M$ by construction, while $x_n$ divides $f$, and hence divides every monomial generator of $fG$. Now by $(\dagger)$, we have that $M\notin I$, which finishes the proof of (ii).
\end{proof}

\section{Saturation and the associated graded ring}\label{sec-gr}

In this section we study the observe the relationship between saturation and the natural map from a local ring $R$ to its associated graded ring $\gr_{\mathfrak{m}}(R)$. 

\subsection{Saturation of Initial Ideals}
When $\gr_{\mathfrak{m}}(R)$ is an integral domain, this map is multiplicative, but not in general additive, so not a ring homomorphism. Throughout this section we shall fix a local domain $(R, \mathfrak m,k)$. The map is defined as follows.

\begin{Comments}(See \cite{VP-VG_gr_paper})
    For $a\in R$, we define the order of $a$
    \begin{equation*}
        v(a):=
        \sup\{r\in \mathbb{N}\mid a\in \mathfrak m^r\}.
    \end{equation*}
    In particular $v(1) = 0$ and $v(0) = \infty$ and for $a\neq 0$ we call the element $a + \mathfrak{m}^{v(a)+1}\in \gr_{\mathfrak{m}}(R)$ the \emph{initial form of $a$} (we set $0^* = 0\in \gr_{\mathfrak{m}}(R)$). $v$ extends to a set map $Q(R)\to \mathbb{Z}$, where $Q(R)$ is the quotient field of $R$. It is well-known that $v:Q(R)\to \mathbb{Z}$ is generally not a valuation (see for example \cite[Chapter VIII Section 1]{ZS2}), but it is a valuation if $\gr_{\mathfrak{m}}(R)$ is a domain (c.f. \cite[Chapter VIII Section 1 Theorem 1]{ZS2}).\\

    It is also well-known that the natural map
    \begin{equation*}
        \begin{split}
            R&\to \gr_{\mathfrak{m}}(R)\\
            a&\mapsto a^* = a+ \mathfrak{m}^{v(a)+1}
        \end{split}
    \end{equation*}
    is injective.
\end{Comments}

In order to study saturation in $\gr_{\mathfrak{m}}(R)$, and obtain a suitable analogue of our equation (\ref{eq-main}), we define the initial ideal in $\gr_{\mathfrak{m}}(R)$ of an ideal in $R$.

\begin{Definition}(c.f. \cite{VP-VG_gr_paper})
    Let $I\subset R$ be an ideal. The \emph{initial ideal} of $I$ is the ideal $I^*$ in $\gr_{\mathfrak{m}}(R)$ generated by the initial forms of elements in $I$. That is,
    \begin{equation*}
        I^* = (a^*\mid a\in I).
    \end{equation*}

 Equivalently, $I^*$ is the ideal $\bigoplus_{n=0}^\infty \dfrac{(I\cap \mathfrak{m}^n)+\mathfrak m^{n+1}}{\mathfrak{m}^{n+1}}$ in $\gr_{\mathfrak{m}}(R)$. 
\end{Definition}

We may define saturation in $\gr_{\mathfrak{m}}(R)$ naturally by coloning with the graded maximal ideal $\mathfrak{M}:= \oplus_{i>0}\mathfrak{m}^i/\mathfrak{m}^{i+1}$. More precisely, for an ideal $J\subset \gr_{\mathfrak{m}}(R)$, the saturation of $J$ is the ideal
\begin{equation*}
    J^{\sat}:= J:\mathfrak{M}^{\infty}
    = \cup_{i>0} J: \mathfrak{M}^i.
\end{equation*}
We ask the following general question and provide a partial answer.

\begin{Question}\label{question-gr}
    When do we have the relation
    \begin{equation*}\label{eq-gr}
        (I^{\sat})^*= (I^*)^{\sat}
    \end{equation*}
    for all ideals in $R$?
\end{Question}

The formula (\ref{eq-gr}) fails in general, e.g. by \cite[Example 3.10]{CPM}. However, we show that one inclusion holds in the following proposition, under a mild assumption.

\begin{Proposition}\label{prop-gr1}
    Suppose that $\gr_{\mathfrak{m}}(R)$ is a domain. Then for any ideal $I\subset R$, we have the inclusion
    \begin{equation*}
        (I^{\sat})^*\subset
        (I^*)^{\sat}.
    \end{equation*}
\end{Proposition}
\begin{proof}
    We begin by noting that for $t>0$,
    \begin{equation*}
        \mathfrak{M}^t = (\oplus_{i\geq 1}\mathfrak{m}^i/\mathfrak{m}^{i+1})^t
        =
        \oplus_{i\geq t}\mathfrak{m}^i/\mathfrak{m}^{i+1}.
    \end{equation*}
    This combined with the fact that $\gr_{\mathfrak{m}}(R)$ is generated in degree one, gives that for any ideal $J\subset \gr_{\mathfrak{m}}(R)$. we have
    \begin{equation*}
        J:\mathfrak{M}^t = J:(\mathfrak{m}^t/\mathfrak{m}^{t+1}) = \{f\in \gr_{\mathfrak{m}}(R) \mid f(\mathfrak{m}^t/\mathfrak{m}^{t+1})\subset J\}.
    \end{equation*}
    Let $t$ be a positive integer such that $I^{\sat} = I:\mathfrak{m}^t$ and simultaneously $(I^*)^{\sat} = I^*: \mathfrak{M}^t$. Consequently, $(I^*)^{\sat} = I: (\mathfrak{m}^t/\mathfrak{m}^{t+1})$.\\

    Let $0\neq f\in (I^{\sat})^*$ be of the form $f = y^*$ such that $y\in I^{\sat} (= I:\mathfrak{m}^t)$. Since $0\neq f = y^*$, we have $y\neq 0$, so that $v(y)\neq \infty$ by Krull's Intersection Theorem. It remains to show that $y^* \mathfrak{m}^t/\mathfrak{m}^{t+1}\subset I^*$. Let $0\neq g\in \mathfrak{m}^t/\mathfrak{m}^{t+1}$. Then $g = x+\mathfrak{m}^{t+1}$ for some $x\in R$ having the property that $v(x) = t$. This combined with our assumption that $\gr_{\mathfrak{m}}(R)$ is a domain yields the following
    \begin{equation*}
        y^*g
        = (y+\mathfrak{m}^{v(y)+1})(x+\mathfrak{m}^{t+1})
        = yx+\mathfrak{m}^{v(y) +t+1}
        =yx + \mathfrak{m}^{v(yx) +1}
        = (yx)^*.
    \end{equation*}
    To finish the proof, it only remains to show that $yx\in I$. On the other hand, $y\in I:\mathfrak{m}^t$ and $g\in \mathfrak{m}^t$, so that $yx\in (I:\mathfrak{m}^t)\mathfrak{m}^t\subset I$ and we are done.
\end{proof}

\begin{Corollary}\label{corr-gr1}
    Suppose that $\gr_{\mathfrak{m}}(R)$ is a domain. Then for any ideal in $R$
    \begin{equation*}
        [(I^{\sat})^*]^{\sat} = (I^*)^{\sat}
    \end{equation*}
\end{Corollary}
\begin{proof}
    By applying saturation to both sides of the inclusion of Proposition \ref{prop-gr1}, we obtain
    \begin{equation*}
        [(I^{\sat})^*]^{\sat}\subset [(I^*)^{\sat}]^{\sat} = (I^*)^{\sat}.
    \end{equation*} Then since $(-)^*$ preserves inclusion,
    \begin{equation*}
        [(I^{\sat})^*]^{\sat}\subset (I^*)^{\sat} \subset [(I^{\sat})^*]^{\sat}.
    \end{equation*}
\end{proof}

\subsection{Generalized Monomial Ideals}

In this subsection we show that the commuting of saturation and taking the initial ideal in the graded ring always holds, at least when the base ring is regular, see Theorem \ref{thm gr}.

\begin{Definition}
    Let $S$ be a standard graded ring with homogeneous maximal ideal $\mathfrak{ M}$, then we define
    \begin{equation*}
        \depth (S) = \grade_{\mathfrak{M}}(S).
    \end{equation*}
 \end{Definition}

In the following, given a local ring $(R,\mathfrak m)$ and an $R$-regular sequence $\mathbf x$, we will use the notion of monomials and monomial ideals in $\mathbf x$ in the sense introduced in \cite[Definition 2.1]{genmon}. We also recall that for any regular local ring, if $\mathbf x$ is a minimal generating set for $\mathfrak m$, then $\mathbf x$ is an $R$-regular sequence.

\begin{Lemma}\label{lem-genmon}
    Let $(R,\mathfrak m)$ be a regular local ring, $\mathbf x=\{x_1,\ldots,x_d\}$ be a minimal generating set of $\mathfrak m$ and $I$ be a monomial ideal in $\mathbf{x}$. Then, $I^*$ is a monomial ideal generated by $f^*$ as $f$ varies over any monomial generating set of $I$. 
\end{Lemma} 

\begin{proof}
    Let $f_1,\ldots,f_r$ be a monomial generating set of $I$ and let $f\in I\setminus \{0\}$. Then $f = \sum_{1\leq i\leq r}c_if_i$ for some nonzero $c_i\in R$. We have that $c_i = \sum_j a_{i,j}M_{i,j}$ for some monomials $M_{i,j}$ in $\mathbf{x}$ and nonzero $a_{i,j}\in R^{\times}$. Now after suitably reindexing the $f_i$, $f$ admits an expression $f = \sum_{1\leq j\leq p,1\leq i\leq s}a_{i,j}M_{i,j}f_i$ for some $s\leq r$, where all of the monomials $M_{i,j}f_i$, $1\leq i\leq s$, $1\leq j\leq p$, are distinct. 
    
    Let $v$ be the smallest $\mathfrak{m}$-adic order of all the monomials $M_{i,j}f_i$, which is well-defined since $R$ is regular implies that this order is a valuation. For a monomial $a\in R$ let $\text{ord}(a)$ denote this order.
    \begin{equation*}
        \begin{split}
            f^* &= \left(\sum_{1\leq j\leq p,1\leq i\leq s}a_{i,j}M_{i,j}f_i\right)^*
        = \left(0,\ldots,0, \sum_{1\leq j\leq p,1\leq i\leq s}a_{i,j}M_{i,j}f_i,0,\ldots\right)
        \\&=\sum_{1\leq j\leq p,1\leq i\leq s, \text{ord}(M_{i,j}f_i) = v}(a_{i,j}M_{i,j}f_i)^*
        =\sum_{1\leq j\leq p,1\leq i\leq s, \text{ord}(M_{i,j}f_i) = v}a_{i,j}^*M_{i,j}^*f_i^*.
        \end{split}
    \end{equation*}
    Therefore, $I^*$ is generated by the initial forms $f_1^*,\ldots,f_r^*$, which completes the proof.
\end{proof}

\begin{Lemma}\label{lem-star}
Let $(R,\mathfrak m)$ be a regular local ring, $\mathbf x=\{x_1,\ldots,x_d\}$ be a minimal generating set of $\mathfrak m$ and $A$, $B\subset R$ monomial ideals in $\mathbf{x}$. Then
\begin{equation*}
    A^*\cap B^*\subset (A\cap B)^*.
\end{equation*}
\begin{proof}
    Let $k = R/\mathfrak m$ and for $f = \prod_{1\leq i\leq d}x_i^{r_i}$ let $\text{ord}_{x_i}f:= r_i$. Then $\gr_{\mathfrak m}(R)$ is (isomorphic to) the polynomial ring $k[x_1,\ldots,x_d]$. Then by Lemma \ref{lem-genmon}, the monomial ideal $A^*\cap B^*$ is generated by the set $\{\gcd(F,G) \mid F\in  A^*,G\in B^*$. Now by \cite[Lemma 3]{genmon}, it is enough to show that for $f\in A$ and $g\in B$,
    \begin{equation*}
        \text{lcm}(f,g)^* = \text{lcm}(f^*,g^*),
    \end{equation*}
    where $\text{lcm}(f,g):= \prod_{1\leq i\leq d}x_i^{\max(\text{ord}_{x_i}(f),\text{ord}_{x_i}(g))}$, but this holds since $(ab)^* = a^*b^*$ for all $a,b\in R$.
\end{proof}
    
\end{Lemma}

\begin{Theorem}\label{thm gr}
    Let $(R,\mathfrak m)$ be a regular local ring, $\mathbf x=\{x_1,\ldots,x_d\}$ be a minimal generating set of $\mathfrak m$ and $I\subset R$ be a monomial ideal in $\mathbf x$. Then
    \begin{equation*}
        (I^*)^{\sat} = (I^{\sat})^*.
    \end{equation*}
In particular, $\operatorname{depth} \operatorname{gr}_{\mathfrak m}(R/I^{\sat})>0$. 
\end{Theorem}
\begin{proof}
    By Proposition \ref{prop-gr1}, it remains to show that $(I^*)^{\sat}\subset (I^{\sat})^*$. Fix $t>0$ such that $I^{\sat} = I:\mathfrak{m}^t$ and $(I^*)^{\sat} = I^*:\oplus_{i\geq t}\mathfrak{m}^i/\mathfrak{m}^{i+1}$. Then by Lemma \ref{lem-genmon} and Lemma \ref{satlemma} (which is also valid for standard graded rings such as $\gr_{\mathfrak{m}}(R)$), $I^* = \cap_{1\leq j\leq d}(I^*: ([x_j^*]^t))$.
    
    Let $f_1,\ldots,f_r$ be monomial generators of $I$ in $\mathbf{x}$, so that by Lemma \ref{lem-genmon}, $I^* = (f_1^*,\ldots,f_r^*)$. By Lemma \ref{satlemma}, we have $I^{\sat} = \cap_{1\leq j\leq d} (f_1,\ldots,f_r): (x_j^t)$, which equals $\cap_{1\leq j\leq d}\sum_{1\leq i\leq r} \text{lcm}(f_i,x_j^t)/(x_j^t)$ by \cite[Proposition 1 on page 489]{genmon}. Consequently, by Lemma \ref{lem-star}
    \begin{equation*}
        \begin{split}
            (I^*)^{\sat} &= \bigcap_{1\leq j\leq d}(f_1^*,\ldots,f_r^*): ([x_j^*]^t)
            =\bigcap_{1\leq j\leq d}\sum_{1\leq i\leq r}(f_i^*): ([x_j^*]^t)
            \\&=\bigcap_{1\leq j\leq d}\sum_{1\leq i\leq r}\text{lcm}(f_i^*,[x_j^*]^t)/[x_j^*]^t
            =\bigcap_{1\leq j\leq d}\sum_{1\leq i\leq r}\text{lcm}(f_i^*,[x_j^t]^*)/[x_j^t]^*
            \\&=
            \bigcap_{1\leq j\leq d}\sum_{1\leq i\leq r}[(\text{lcm}(f_i,[x_j^t])/[x_j^t])^*]
            \subset
            \bigcap_{1\leq j\leq d}\left(\left(\sum_{1\leq i\leq r}(\text{lcm}(f_i,[x_j^t])/[x_j^t])\right)^*\right)
            \\&
            \subset
            \left(\bigcap_{1\leq j\leq d}\sum_{1\leq i\leq r}(\text{lcm}(f_i,[x_j^t])/[x_j^t])\right)^* = (I^{\sat})^*
        \end{split}
    \end{equation*}
    (the first non-equality inclusion holds, since for a set $A$ of monomials in $\textbf{x}$, the proof of Lemma \ref{lem-genmon} shows that $\left(\sum_{f\in A}f\right)^* = \sum_{f\in A}(f^*)$ ). This completes the proof of the theorem.
\end{proof}

Theorem \ref{thm gr} states that the formula (\ref{eq-main}) holds for any ideal in a regular local ring where $S = \gr_{\mathfrak{m}}(R)$, in contrast to the dependence on $R$ when $S$ is a quotient of $R$ of Theorem \ref{thmA} or a flat injective codomain of $R$ in Theorem \ref{thmB}.

\subsection*{Acknowledgments}

The authors would like to thank Ian Aberbach, Dale Cutkosky, Rankeya Datta, and Sudipta Das for their advice and helpful comments. The authors would also like to thank Anurag Singh for bringing the ring of Example \ref{ex1} to our attention. The second author thanks the Center for Mathematical Sciences and Applications for their support during the writing of this manuscript. The second author was partially supported by the Israel Science Foundation grant ISF-687/24.


\bibliographystyle{amsplain}

\bibliography{biblio}

\end{document}